\documentclass[10pt,reqno]{amsart}

\usepackage[utf8]{inputenc}
\usepackage[english]{babel}
\usepackage[dvipsnames]{xcolor}
\usepackage{amssymb}
\usepackage{hyperref}
\usepackage{mathrsfs}

\usepackage{xparse}
\NewDocumentCommand{\ceil}{s O{} m}{%
  \IfBooleanTF{#1} 
    {\left\lceil#3\right\rceil} 
    {#2\lceil#3#2\rceil} 
}

\usepackage{color}
\usepackage{enumerate}

\newtheorem{theorem}{Theorem}[section]
\newtheorem{lemma}[theorem]{Lemma}
\newtheorem{corollary}[theorem]{Corollary}

\theoremstyle{remark}
\newtheorem{remark}[theorem]{\bf{Remark}}

\theoremstyle{definition}
\newtheorem{assumption}[theorem]{Assumption}

\newtheorem{definition}[theorem]{Definition}

\newcommand\cbrk{\text{$]$\kern-.15em$]$}}
\newcommand\opar{\text{\,\raise.2ex\hbox{${\scriptstyle
|}$}\kern-.34em$($}}
\newcommand\cpar{\text{$)$\kern-.34em\raise.2ex\hbox{${\scriptstyle |}$}}\,}

\newcommand{\aint}{-\hspace{-0.38cm}\int}
\newcommand{\aaint}{-\hspace{-0.31cm}\int}

\newcommand\bC{\mathbb{C}}

\newcommand\bE{\mathbb{E}}

\newcommand\bH{\mathbb{H}}

\newcommand\bL{\mathbb{L}}

\newcommand\bN{\mathbb{N}}

\newcommand\bP{\mathbb{P}}
\newcommand\bQ{\mathbb{Q}}
\newcommand\bR{\mathbb{R}}

\newcommand\bZ{\mathbb{Z}}

\newcommand\cF{\mathcal{F}}
\newcommand\cG{\mathcal{G}}

\newcommand\cK{\mathcal{K}}
\newcommand\cL{\mathcal{L}}
\newcommand\cM{\mathcal{M}}

\newcommand\cO{\mathcal{O}}

\newcommand\cS{\mathcal{S}}

\newcommand{\mysection}[1]{\section{#1}
\setcounter{equation}{0}}

\newcommand{\Ccinf}{C_{c}^{\infty}}
\newcommand{\R}{\mathbb{R}}

\begin{document}

\title[PDE with space-time fractional operators]
{An $L_q(L_p)$-theory for diffusion equations with space-time nonlocal operators}

\author{Kyeong-Hun Kim}
\address{Department of Mathematics, Korea University, 145 Anam-ro, Seongbuk-gu, Seoul,
02841, Republic of Korea}
\email{kyeonghun@korea.ac.kr}

\thanks{The authors were  supported by the National Research Foundation of Korea(NRF) grant funded by the Korea government(MSIT) (No. NRF-2020R1A2C1A01003354)} 

\author{Daehan Park}
\address{Department of Mathematics, Korea University, 145 Anam-ro, Seongbuk-gu, Seoul,
02841, Republic of Korea} \email{daehanpark@korea.ac.kr}

\author{Junhee Ryu}
\address{Department of Mathematics, Korea University, 145 Anam-ro, Seongbuk-gu, Seoul,
02841, Republic of Korea} \email{junhryu@korea.ac.kr}



\begin{abstract}

We present an $L_q(L_{p})$-theory for the  equation
$$
\partial_{t}^{\alpha}u=\phi(\Delta) u +f, \quad t>0,\, x\in \mathbb{R}^d \quad\, ;\, u(0,\cdot)=u_0.
$$
Here $p,q>1$,  $\alpha\in (0,1)$, $\partial_{t}^{\alpha}$ is the Caputo fractional derivative of order $\alpha$, and $\phi$ is a Bernstein function satisfying the following: $\exists \delta_0\in (0,1]$ and $c>0$  such that 
\begin{equation}
  \label{eqn 8.17.1}
c \left(\frac{R}{r}\right)^{\delta_0}\leq \frac{\phi(R)}{\phi(r)}, \qquad  0<r<R<\infty.
\end{equation}
We prove uniqueness and existence results in Sobolev spaces, and  obtain  maximal regularity results of the solution. In particular, we prove
\begin{align*}
\| |\partial^{\alpha}_t u|+|u|+|\phi(\Delta)u|\|_{L_q([0,T];L_p)}\leq N(\|f\|_{L_q([0,T];L_p)}+ \|u_0\|_{B_{p,q}^{\phi,2-2/ \alpha q}}),
\end{align*}
where $B_{p,q}^{\phi,2-2/\alpha q}$ is a modified Besov space on $\mathbb{R}^d$  related to $\phi$.

Our approach is based on BMO estimate  for $p=q$ and vector-valued Calder\'on-Zygmund theorem for $p\neq q$. The Littlewood-Paley theory is also used to treat the non-zero initial data problem. Our proofs rely on the derivative estimates of the fundamental solution, which are  obtained in this article based on   the probability theory.
\end{abstract}

\maketitle

\mysection{Introduction}

Many types of diffusion equations have  been used to describe diverse phenomena in various fields including mathematics, engineering, biology, hydrology, finance, and chemistry.
The classical heat equation $\partial_t u = \Delta u$ describes the heat propagation in homogeneous media. When $\alpha\in(0,1)$, the equation $\partial_t^\alpha u = \Delta u$ describes the anomalous diffusion exhibiting subdiffusive behavior caused by particle sticking and trapping effects
 (e.g. \cite{metzler1999anomalous,metzler2000random}). 
 On the other hand, the spatial nonlocal operator $\phi(\Delta)$ describes long range jumps of particles,  diffusions  on fractal structures, and  long time behavior of particles moving in space with quenched and disordered force field (e.g. \cite{bouchaud1990,fogedby1994}). 

The space-time fractional diffusion equation  can be used to describe the combined  phenomena, for instance, jump diffusions with a higher peak 
and heaver tails  (see e.g. \cite{chen2012space,gorenflo2003CTRW}). The space-time fractional equation is also related to the scaling limit of continuous time random walk  (see \cite{compte1996, gorenflo2003CTRW, meerschaert2002stochastic}).

In this article we study  the space-time fractional equation
\begin{equation} \label{main equation intro}
\partial_t^\alpha u=\phi(\Delta)u + f, \quad t>0\,;\quad u(0,\cdot)=u_0.
\end{equation}
Here, $\alpha\in(0,1), p,q>1$, $\partial_t^\alpha$ is the Caputo fractional derivative of order $\alpha$, and $\phi$ is a Bernstein function satisfying $\phi(0+)=0$, that is, $\phi:(0,\infty) \to (0,\infty)$ such that 
$$
(-1)^n \phi^{(n)}(\lambda)\leq 0, \quad \forall \lambda>0, \, n=1,2,\cdots.
$$
The operator $\phi(\Delta)$ is defined by 
$$
\phi(\Delta)u:=-\phi(-\Delta)u:=-\cF^{-1}_d[\phi(|\xi|^2)\cF_d (u)(\xi)], \quad u\in C^{\infty}_c.
$$
For instance, if $\phi(\lambda)=\lambda^{\nu/2}$, $\nu\in (0,2)$,  then $\phi(\Delta)=\Delta^{\nu}$ becomes the fractional Laplacian.  It turns out  that $\phi(\Delta)$ is a type of  integro-differential operator, and the class of $\phi(\Delta)$ is characterized by  the infinitesimal generators of  subordinate Brownian motions. See Section 2 for details.

Probabilistic representation of solution to equation \eqref{main equation intro} has been introduced  e.g.  in \cite{ chen2017time, chen2020time, chen2012space, meerschaert2002stochastic}. Actually the transition density of  subordinate Brownian motion delayed by an inverse subordinator becomes the fundamental solution, and this transition density  and its appropriate time-fractional derivative appear in the solution representation. See Sections 4 and 5.

The main purpose of this article is to present a Sobolev-regularity theory of equation  \eqref{main equation intro}. 
We prove  the uniqueness and existence   in Sobolev spaces  and  obtain the maximal $L_q(L_p)$-regularity of solutions. In particular, we prove  for any $p,q>1$, 
\begin{equation}
    \label{eqn 8.4.5}
\| |\partial^{\alpha}_t u|+|u|+|\phi(\Delta)u|\|_{L_q([0,T];L_p)}\leq N\|f\|_{L_q([0,T];L_p)}+N  \|u_0\|_{B_{p,q}^{\phi,2-2/ \alpha q}},
\end{equation}
where $B_{p,q}^{\phi,2-2/\alpha q}$ is a  Besov space on $\bR^d$ related to $\phi$.  Moreover, we obtain the maximal regularity  of higher order  derivatives as well as negative order derivatives of  solutions. 

 Our proof for \eqref{eqn 8.4.5} is mainly based on BMO-estimate if $u_0=0$,  and  Littlewood-Paley theory is used  to treat the case $u_0\neq 0$.  Specifically speaking,  we prove that if $u_0=0$ then
 \begin{equation}
 \label{eqn 8.17.3}
 |(\partial^{\alpha}_tu)^{\#}(t,x)|+  |(\phi(\Delta)u)^{\#}(t,x)| \leq N \|f\|_{L_{\infty}}, \quad t>0,\,x\in \bR^d.
 \end{equation}
 Here $(\partial^{\alpha}_tu)^{\#}$ denotes the sharp function of $\partial^{\alpha}_tu$. The BMO estimate and the Marcinkiewicz interpolation theorem lead to \eqref{eqn 8.4.5} for $p=q$, and the case $p\neq q$ is covered based on the vector-valued  Calder\'on-Zygmund theorem.  
 For the implement of  these procedures, we rely on sharp upper bounds of arbitrary  order space-time derivatives of the fundamental solution, which are obtained in Section 3.  Due to the non-integrability of derivatives of the fundamental solution, our proofs of e.g. \eqref{eqn 8.17.3} are much more delicate than the proofs for  PDEs with local operators. Condition \eqref{eqn 8.17.1} is a minimal assumption on $\phi$  such that our derivative estimates of the fundamental solution hold for all $t>0$.  This is essential in this article  because  we are aiming to prove estimates for solutions which are independent of the time intervals where the solutions are defined.
 
Here are some related $L_q(L_p)$-theories for the diffusion equations with either time fractional derivative  or 
spatial integro-differential operators. An $L_q(L_p)$-theory for the time fractional equation 
$$
\partial^{\alpha}_t u =\sum_{i,j=1}^d a^{ij}(t,x)u_{x^ix^j}+f
$$
was  introduced in \cite{clement1992global, Pr1991} when $a^{ij}=\delta^{ij}$. The result of \cite{clement1992global, Pr1991} is based on semigroup theory, and similar approach is used in \cite{zacher2005maximal} to treat the equation with uniformly continuous coefficients.  Recently, the continuity condition of \cite{zacher2005maximal} is significantly relaxed  in \cite{dong2019lp,kim17timefractionalpde}. For instance, if $p=q$ then \cite{dong2019lp} only requires that the coefficients are only measurable in $t$ and have small mean oscillation in  $x$. The approach in \cite{dong2019lp} is based on the level set arguments. 
Regarding the equations with spatial integro-differential  operators, an $L_{p}$-theory of the diffusion equation the type
$$
 u_t=\int_{\R^{d}}\left(u(x+y)-u(x)-\chi(y)y \cdot \nabla u(y)\right) J(x, dy) +f
$$
was introduced in \cite{mik1992}.  Here $\chi$ is a certain indicator function and the jump kernel $J(x,dy)$ is  of the type $a(x,y)|y|^{-d-\alpha}$, where $a(x,y)$ is homogeneous of order zero and sufficiently smooth in $y$.  Recently, the condition on $J(x,dy)$ has been generalized and weaken e.g. in
 \cite{kim2019L, kim16,mikulevivcius2019cauchy, mikulevivcius2017p,  zhang2013p}. 
 
This article is organized as follows. In Section 2, we introduce some basic facts on the fractional calculus, integro-differential operator $\phi(\Delta)$, and related function spaces.  We also introduce our main result, Theorem \ref{main theorem},  in Section 2.  In Section 3 we obtain sharp upper bounds of space-time  derivatives of the fundamental solution.  In section 4 we study the zero initial data problem, and non-zero initial data problem is considered in Section 5.  Finally we prove our main result  in Section 6.

We finish the introduction with some notations. We use $``:="$ or $``=:"$ to denote a definition. The symbol $\bN$ denotes the set of positive integers and $\bN_0:=\bN\cup\{0\}$. Also we use $\bZ$ to denote the set of integers.  As usual $\bR^d$ stands for the Euclidean space of points $x=(x^1,\dots,x^d)$. We set
$$
B_r(x):=\{y\in \bR: |x-y|<r\}, \quad \bR_+^{d+1} := \{(t,x)\in\bR^{d+1} : t>0 \}.
$$
For $i=1,\ldots,d$,
multi-indices $\sigma=(\sigma_{1},\ldots,\sigma_{d})$,
 and functions $u(t,x)$ we set
$$
\partial_{x^{i}}u=\frac{\partial u}{\partial x^{i}}=D_{i}u,\quad D^{\sigma}u=D_{1}^{\sigma_1}\cdots D_{d}^{\sigma_d}u,\quad|\sigma|=\sigma_{1}+\cdots+\sigma_{d}.
$$
We also use the notation $D_{x}^{m}$ for arbitrary partial derivatives of
order $m$ with respect to $x$.
For an open set $\cO$ in $\bR^{d}$ or $\bR^{d+1}$, $C_c^\infty(\cO)$ denotes the set of infinitely differentiable functions with compact support in $\cO$. By 
$\cS=\cS(\bR^d)$ we denote the  class of Schwartz functions on $\bR^d$.
For $p> 1$, by $L_{p}$ we denote the set
of complex-valued Lebesgue measurable functions $u$ on $\R^{d}$ satisfying
\[
\left\Vert u\right\Vert _{L_{p}}:=\left(\int_{\R^{d}}|u(x)|^{p}dx\right)^{1/p}<\infty.
\]
Generally, for a given measure space $(X,\mathcal{M},\mu)$, $L_{p}(X,\cM,\mu;F)$
denotes the space of all $F$-valued $\mathcal{M}^{\mu}$-measurable functions
$u$ so that
\[
\left\Vert u\right\Vert _{L_{p}(X,\cM,\mu;F)}:=\left(\int_{X}\left\Vert u(x)\right\Vert _{F}^{p}\mu(dx)\right)^{1/p}<\infty,
\]
where $\mathcal{M}^{\mu}$ denotes the completion of $\cM$ with respect to the measure $\mu$.
If there is no confusion for the given measure and $\sigma$-algebra, we usually omit the measure and the $\sigma$-algebra. 
We denote $a\wedge b := \min\{a,b\}$ and $a\vee b:=\max\{a,b\}$. By $\cF$ and $\cF^{-1}$ we denote the $d$-dimensional Fourier transform and the inverse Fourier transform respectively, i.e.
$$
\cF(f)(\xi):=\hat{f}(\xi):=\int_{\bR^d} e^{-i\xi\cdot x} f(x)dx, \quad \cF^{-1}(f)(\xi):=\frac{1}{(2\pi)^d}\int_{\bR^d} e^{i\xi\cdot x} f(x)dx. 
$$
For any $a,b>0$, we write $a\sim b$ if there is a constant $c>1$ independent of $a,b$ such that $c^{-1}a\leq b\leq ca$. Finally if we write $N=N(\dots)$, this means that the constant $N$ depends only on what are in the parentheses. The constant $N$ can differ from line to line.

\mysection{Main results} \label{Main result sec}

First we introduce some definitions and facts related to the fractional calculus. For $\alpha>0$ and $\varphi\in L_{1}((0,T))$, the Riemann-Liouville fractional integral
of the order $\alpha$ is defined as
$$
I_{t}^{\alpha}\varphi:=\frac{1}{\Gamma(\alpha)}\int_{0}^{t}(t-s)^{\alpha-1}\varphi(s)ds, \quad 0\leq t\leq T.
$$
We also define  $I^0\varphi:=\varphi$.
By Jensen's inequality, for $p\in[1,\infty]$,
\begin{equation}
                     \label{eq:Lp continuity of I}
\left\Vert I_{t}^{\alpha}\varphi\right\Vert _{L_{p}((0,T))}\leq
N(T,\alpha)\left\Vert \varphi\right\Vert _{L_{p}((0,T))}.
\end{equation}
 Using  Fubini's theorem, one can  easily check for any $\alpha,\beta\geq 0$,
 \begin{equation}
                                                          \label{eqn 4.15.3}
I^{\alpha}_tI^{\beta}_t \varphi=I^{\alpha+\beta}_t \varphi, \quad
\text{$(a.e.)$} \,\, t\leq T.
\end{equation}
 Let  $\alpha \in [n-1, n)$,  $n\in \bN$. 
If  $\varphi(t)$ is  $(n-1)$-times  differentiable and $\left(\frac{d}{dt}\right)^{n-1} I_t^{n-\alpha}  \varphi$ is absolutely continuous on $[0,T]$, then
the Riemann-Liouville fractional derivative $D_{t}^{\alpha}$ and the Caputo fractional derivative $\partial_{t}^{\alpha}$ are defined as
\begin{equation}
                          \label{eqn 4.15}
D_{t}^{\alpha}\varphi:=\left(\frac{d}{dt}\right)^{n}\left(I_{t}^{n-\alpha}\varphi\right),
\end{equation}
and
\begin{align}
         \label{eqn 8.30}
\partial_{t}^{\alpha}\varphi= D_{t}^{\alpha} \left(\varphi(t)-\sum_{k=0}^{n-1}\frac{t^{k}}{k!}\varphi^{(k)}(0)\right).
\end{align}
Note that $D^{\alpha}_t\varphi=\partial^{\alpha}_t \varphi$ if $\varphi(0)=\varphi'(0)=\cdots=\varphi^{(n-1)}(0)=0$.
By \eqref{eqn 4.15.3} and \eqref{eqn 4.15},  if  $\alpha,\beta\geq 0$,
\begin{equation*}
                \label{eqn 4.20.1}
D^{\alpha}_tD^{\beta}_t=D^{\alpha+\beta}_t, \quad D^{\alpha}_t I_{t}^{\beta} \varphi=
D_{t}^{\alpha-\beta}\varphi,
\end{equation*}
where $D_t^{a}\varphi:=I_t^{-a}\varphi$ if $a<0$.
Also if 
$\varphi(0)=\varphi^{(1)}(0)=\cdots = \varphi^{(n-1)}(0)=0$ 
then
\begin{equation}
	\label{eqn 07.19.15:09}
I^{\alpha}_{t}\partial^{\alpha}_{t}u=I^{\alpha}_{t}D^{\alpha}_{t}u=u.
\end{equation}

Next, we introduce our assumption on $\phi$ and some informations  on the operator $\phi(\Delta)$.  Recall that a fucntion $\phi: \bR_+ \to \bR_+$ satisfying  $\phi(0+)=0$ is a Bernstein function if there exist a constant $b\geq 0$ and a L\'evy measure $\mu$ (i.e.  $\int_{(0,\infty)} (1\wedge t) \mu(dt)<\infty$) such that
\begin{equation}
	\label{eqn 07.17.16:36}
\phi(\lambda)=b\lambda + \int_{(0,\infty)} (1-e^{-\lambda t})\mu(dt).
\end{equation}
It is known that a function $\phi$ is a Bernstein function if and only if it is a Laplace exponent of a subordinator, that is, there exists a nonnegative real-valued L\'evy process $S_t$ on a probability space $(\Omega, \mathscr{F}, \bP)$ such that 
$$\bE\, e^{-\lambda S_t}:=\int_{\Omega} e^{-\lambda S_t(\omega)} \,\bP(d\omega)=e^{-t\phi(\lambda)}.
$$

For $f\in\cS(\bR^d)$,  we define $\phi(\Delta)f:=-\phi(-\Delta)f$ as 
\begin{equation*}\label{fourierdefphi}
\phi(\Delta)f(x)=\cF^{-1}(-\phi(|\xi|^{2})\cF(f)(\xi))(x).
\end{equation*}
It turns out (see \cite[Theorem 31.5]{sato1999}) that  $\phi(\Delta)$  is an integro-differential operator defined by
\begin{equation*}\label{fourier200408}
\phi(\Delta)f(x)=b\Delta f+ \int_{\bR^d} \left(f(x+y)-f(x)-\nabla f(x)\cdot y \mathbf{1}_{|y|\leq 1}\right) J(y) dy
\end{equation*}
where $J(x)=j(|x|)$ and $j:(0,\infty)\to(0,\infty)$ is given by
$$
j(r)=\int_{(0,\infty)} (4\pi t)^{-d/2} e^{-r^2/(4t)} \mu(dt).
$$
Furthermore, $\phi(\Delta)$ is the infinitesimal generator of the $d$-dimensional subordinate Brownian motion $X_t:=W_{S_{t}}$;
$$
\phi(\Delta)f(x)=\lim_{t \to 0} \frac{ \bE f(x+X_t) -f(x)}{t},
$$
where  $W_t$ is a $d$-dimensional Brownian motion independent of $S_t$. For instance, by taking $\phi(\lambda)=\lambda^{\alpha/2}$, $\alpha\in (0,2)$, we get the fractional Laplacian $\Delta^{\alpha/2}=-(-\Delta)^{\alpha/2}$, which is the infinitesimal generator of a rotationally symmetric $\alpha$-stable process in $\bR^d$.

Using \eqref{eqn 07.17.16:36} one can  check
$$
\phi^{(n)}(\lambda)=b\mathbf{1}_{n=1}-\int_{(0,\infty)} (-t)^{n}e^{-\lambda t}\mu(dt),\quad n\in \bN,
$$
where $\phi^{(n)}$ is the $n$-th derivative of $\phi$.  Theorefore, $\phi'(\lambda)>0$ and
$$
(-1)^n \phi^{(n)}(\lambda)\leq 0, \quad \forall \lambda>0, \, n\in \bN,
$$
and, by the inequality  $t^{n}e^{-t}\leq N(n)(1-e^{-t})$, we also have for any $n\geq 1$,
\begin{align}
	\label{eqn 07.19.14.35}
\lambda^{n}|\phi^{(n)}(\lambda)| &\leq \mathbf{1}_{n=1}b\lambda+ \int_{0}^{\infty}(\lambda t)^{n}e^{-\lambda t}\mu(dt) \nonumber
\\
& \leq \mathbf{1}_{n=1}b\lambda+ N\int_{0}^{\infty}(1-e^{-\lambda t})\mu (dt)  \nonumber
\\
&\leq N(n) \phi(\lambda).
\end{align}

Here is our assumption on $\phi$.

\begin{assumption}
   \label{ass bernstein}
There exist constants $\delta_0\in (0,1]$ and $c>0$  such that
\begin{equation}\label{e:H}
c \left(\frac{R}{r}\right)^{\delta_0}\leq\frac{\phi(R)}{\phi(r)}, \qquad 0<r<R<\infty.
\end{equation}
\end{assumption}

Note that the case $\delta_0=1$ is included, and  \eqref{e:H} is assumed to hold for all $0<r<R<\infty$. In the literature, it is common to impose separate conditions on $\phi$ near zero and infinity.
For instance, in \cite{kim2013parabolic,kim2014global,mimica2016heat}, conditions (H1) and (H2) below are used for the study of  the transition density of subordinate Brownian motion:

\textbf{(H1)}: $\exists$ $c_{1},c_{2}>0$, and $0<\delta_{1}\leq\delta_{2}<1$ such that
\begin{equation*}
c_{1}\lambda^{\delta_{1}}\phi(t)\leq \phi(\lambda t)\leq c_{2}\lambda^{\delta_{2}}\phi(t)\quad \lambda\geq1, \quad t\geq1.
\end{equation*}

\textbf{(H2)}: $\exists$ $c_{3}>0$, and $0<\delta_{3}\leq 1$ such that
\begin{equation*}
\phi(\lambda t)\leq c_{3}\lambda^{\delta_{3}}\phi(t) \quad \lambda\leq1, \quad t\leq1.
\end{equation*}

Inequality (3.27) in \cite{kim2013parabolic}  shows that (H1) and (H2) together give
\begin{equation}
    \label{eqn 7.17.1}
\frac{c_1}{c_3} \left(\frac{R}{r}\right)^{\delta_1 \wedge \delta_3}\leq\frac{\phi(R)}{\phi(r)}, \qquad 0<r<R<\infty.
\end{equation}
Thus our assumption is weaker than (H1) combined with (H2).

\vspace{2mm}

Here are some examples related to our assumption on $\phi$:

	\begin{enumerate}[(1)]
\item Stable subordinators : $\phi(\lambda)=\lambda^\beta, \quad 0<\beta\leq1$.

\item Sum of stable subordinators : $\phi(\lambda)=\lambda^{\beta_1}+\lambda^{\beta_2}, \quad 0<\beta_1,\beta_2\leq1$.

\item Stable with logarithmic correction : $\phi(\lambda)=\lambda^\beta (\log(1+\lambda))^\gamma, \quad \beta\in(0,1), \gamma\in(-\beta,1-\beta)$.

\item Relativistic stable subordinators : $\phi(\lambda)=(\lambda+m^{1/\beta})^\beta-m, \quad \beta\in(0,1), m>0$.

\item Conjugate geometric stable subordinators : $\phi(\lambda)=\frac{\lambda}{\log(1+\lambda^{\beta/2})}, \quad \beta\in(0,2)$.
\end{enumerate}

One can check that Examples ($1$)-($4$) satisfy both (H1) and (H2), and  therefore Assumption  \ref{ass bernstein} is also fulfilled due to  \eqref{eqn 7.17.1}. On the other hand, Example ($5$) satisfies Assumption  \ref{ass bernstein} with $\delta_0=1-\frac{\beta}{2}$, but condition (H1) fails to hold because 
$$
\lim_{\lambda \to \infty} \frac{\phi(\lambda)}{\lambda^{1-\varepsilon}}=\infty, \quad \forall  \, \varepsilon>0.
$$

\vspace{2mm}

Next we introduce Sobolev and Besov spaces related to the operator $\phi(\Delta)$.
For $\gamma\in\bR$, denote 
$$H_p^{\phi,\gamma}:=(1-\phi(\Delta))^{-\gamma/2}L_p.
$$
 That is, $u\in H_p^{\phi,\gamma}$ if
\begin{equation*}
\|u\|_{H_p^{\phi,\gamma}}:=\|(1-\phi(\Delta))^{\gamma/2}u\|_{L_p}:=\|\cF^{-1}\{\left(1+\phi(|\cdot|^2)\right)^{\gamma/2}\cF(u)(\cdot)\}\|_{L_p}<\infty.
\end{equation*}
Note that if $\phi(\lambda)=\lambda$, then $H_p^{\phi,\gamma}$ is the  classical Bessel potential space $H^{\gamma}_{p}$. 

The following lemma gives some basic properties of $H^{\phi,\gamma}_{p}$.

\begin{lemma}\label{H_p^phi,gamma space}

(i) For any $\gamma\in\bR$, $H_p^{\phi,\gamma}$ is a Banach space.

(ii) For any $\mu,\gamma\in\bR$, the map $(1-\phi(\Delta))^{\mu/2}: H_p^{\phi,\gamma}\to H_p^{\phi,\gamma-\mu}$ is an isometry.

(iii) If $\gamma_1\leq\gamma_2$, then $H_p^{\phi,\gamma_2}\subset H_p^{\phi,\gamma_1}$, and there is a constant $N>0$ so that
\begin{equation*}
\|u\|_{H_p^{\phi,\gamma_1}}\leq N\|u\|_{H_p^{\phi,\gamma_2}}.
\end{equation*}

(iv) For any  $\gamma\geq0$,
\begin{equation}\label{eqn 03.25.15:03}
\left(\|u\|_{L_p}+\|\phi(\Delta)^{\gamma/2}u\|_{L_p}\right) \sim  \|u\|_{H_p^{\phi,\gamma}}.
\end{equation}

\end{lemma}

\begin{proof}
See \cite[Lemma 6.1]{kim2013parabolic}. We remark that \cite[Lemma 6.1]{kim2013parabolic} is proved for arbitrary Bernstein functions $\phi$ with no drift, that is $b=0$ (see \eqref{eqn 07.17.16:36}). The same proof works for us  because it is proved based on \eqref{eqn 07.19.14.35}.
\end{proof}

Take a function $\Psi\in\cS(\R^{d})$ whose Fourier transform $\hat{\Psi}(\xi)$ is supported in a strip $\{1/2\leq |\xi| \leq 2\}$, $\hat{\Psi}\geq 0$, and 
\begin{equation}\label{eqn 06.23.10:19}
\sum_{j\in\bZ}\hat{\Psi}(2^{-j}\xi)=1,\quad \xi\neq 0.
\end{equation}
Define 
\begin{equation}\label{eqn 06.23.10:17}
\begin{gathered}
\hat{\Psi}_{j}(\xi)=\hat{\Psi}(2^{-j}\xi),\quad j=\pm1,\pm2,\dots,
\\
\hat{\Psi}_{0}(\xi)=1-\sum_{j=1}^{\infty}\hat{\Psi}(\xi).
\end{gathered}
\end{equation}
Also, for  distributions (or  functions) $f$ we define $f_{j}=\Psi_{j}\ast f$, where $\ast$ denotes the convolution.

\begin{definition}(cf. \cite{mikulevivcius2019cauchy}) \label{def 06.17.10:27}
Let $1<p,q<\infty$ and $s\in \bR$. The Besov space $B^{\phi,s}_{p,q}=B^{\phi,s}_{p,q}(\bR^d)$ is a closure of $\cS$ under the norm  
\begin{equation*}\label{eqn 07.17.15:53}
\|u\|_{B^{\phi,s}_{p,q}}:=\left\{ \|u_{0}\|^{q}_{L_{p}}+\sum_{j=1}^{\infty}\phi(    2^{2j} )^{sq/2}\|u_{j}\|^{q}_{L_{p}} \right\}^{1/q}.
\end{equation*}
\end{definition}

\begin{remark}\label{rmk 06.23.11:13}
(i) If $\phi(\lambda)=\lambda$, then $B_{p,q}^{\phi,s}$ corresponds to the standard Besov space $B_{p,q}^s$. In general, since $\phi(\lambda)\leq N\lambda$ for $\lambda\geq 1$, for $u\in \cS$ and  $s \geq 0$ we have
$$
\|u\|_{B^{\phi,s}_{p,q}}\leq N \|u\|_{B^{s}_{p,q}}.
$$
This and \cite[Theorem 6.2.4]{Bergh} yield that  if $\gamma>(s\vee 0)$ then 
\begin{equation}
	\label{eqn 08.05.19:49}
 H^{\gamma}_p \subset B_{p,q}^{\phi,s},
\end{equation}
 and the embedding is continuous. 

(ii) Let $J^{s}u:=(1-\phi(\Delta))^{s/2}u$ and  $\tilde{B}_{p,q}^{\phi,s}$ be the closure of $\cS$ under the norm
$$
\|u\|_{\tilde{B}^{\phi,s}_{p,q}}=\left\{ \sum_{j=0}^{\infty}\|J^{s}\Psi_{j}\ast u\|^{q}_{L_{p}} \right\}^{1/p}.
$$
We now show that
\begin{equation}\label{eqn 06.25.13:31}
\tilde{B}^{\phi,s}_{p,q} = B^{\phi,s}_{p,q}, \quad \quad 1<p,q<\infty, \, s\in \bR.
\end{equation}
First we  prove that there is a constant $N$ such that
\begin{align} \label{multiplier result}
\|J^{s}\Psi_{j}\ast u\|_{L_{p}}  \leq N \phi(  2^{2j}  )^{s/2}\|u_{j}\|_{L_{p}}, \quad \forall j\in\bN_{0}.
\end{align}
 If we denote $\zeta_{j}=\Psi_{j-1}+\Psi_{j}+\Psi_{j+1}$ for $j\in\bN$ and $\zeta_{0}=\Psi_{0}+\Psi_{1}$, then
$$
(1+\phi(|\xi|^2))^{s/2} \hat{\Psi}_{j}(\xi) = (1+\phi(|\xi|^2))^{s/2} \hat{\Psi}_{j}(\xi) \hat{\zeta}_j(\xi).
$$
Fix $j$, and let $\xi$ be  in the support of $\zeta_j$.  Then for any $\nu\in \bR$ and $m\in \bN_{0}$, by \eqref{eqn 07.19.14.35}
$$
\frac{(1+\phi(|\xi|^2))^{\nu}}{\phi(2^{2j})^\nu} \leq N,
$$
$$
\left|\frac{D^{m}_{\xi} (\phi(|\cdot|^2))(\xi)}{\phi(   2^{2j} )^{m}}\right|+\left|D^{m}_{\xi} (\hat{\zeta}_j(\cdot))(\xi)\right| \leq N |\xi|^{-|m|},
$$
where $N$ is independent of $j$. Hence combining the above inequalities, for any $m \in \bN_{0}$ it follows that
$$
\left| D^{m}_{\xi}\left[\frac{ (1+\phi(|\xi|^2))^{s/2} \hat{\Psi}_{j}(\xi) \hat{\zeta}_j(\xi)}{\phi(   2^{2j} )^{s/2}}\right]  \right| \leq N |\xi|^{-|m|},
$$
where $N$ is independent of $j$. Therefore, due to \cite[Theorem 5.2.7]{grafakos2008classical}, \eqref{multiplier result} holds . This implies that for any $s\in\bR$
$$
B^{\phi,s}_{p,q}\subset\tilde{B}^{\phi,s}_{p,q}.
$$
Using \eqref{multiplier result} again, one can easily check 
$$
\|u_{j}\|_{L_{p}}=\|J^{-s}J^s u_{j}\|_{L_{p}} \leq N  \phi(2^{2j})^{-s/2}\|J^s u_{j}\|_{L_{p}},
$$
and therefore  \eqref{eqn 06.25.13:31} holds.

(iii)
Since $(1-\phi(\Delta))^{\nu/2}$ is an isometry from $\tilde{B}^{\phi,s}_{p,q}$ to $\tilde{B}^{\phi,s-\nu}_{p,q}$, we may consider this operator as an isometry from $B^{\phi,s}_{p,q}$ to $B^{\phi,s-\nu}_{p,q}$ due to \eqref{eqn 06.25.13:31}.
\end{remark}

For $p,q\in(1,\infty), \gamma\in\bR$ and $T<\infty$, we denote
\begin{equation*}
\bH_{q,p}^{\phi,\gamma}(T):=L_q\left((0,T); H_p^{\phi,\gamma}\right), \qquad \bL_{q,p}(T):=\bH_{q,p}^{\phi,0}(T).
\end{equation*}
We write $u\in C_{p}^{\alpha,\infty}([0,T]\times\bR^d)$ if $D^m_x u, \partial_t^\alpha D^m_x u \in C([0,T];L_p)$ for any $m\in \bN_{0}$. $C_{p}^{\infty}(\bR^d)$ denotes the set of functions $u_0=u_0(x)$ such that $D^m_xu_0\in L_p$ for any $m\in \bN_{0}$.

\begin{definition}
 \label{defn defining}

(i) For $\alpha\in(0,1)$, $1<p,q<\infty$, and $\gamma\in\bR$, we write $u\in {\bH_{q,p}^{\alpha,\phi,\gamma+2}}(T)$ if there exists a sequence $u_n\in C^{\alpha,\infty}_{p}([0,T]\times \bR^d)$ satisfying
\begin{equation*}
\|u-u_n\|_{\bH_{q,p}^{\phi,\gamma+2}(T)} \to 0 \text{ and } \|\partial_t^\alpha u_n - \partial_t^\alpha u_m \|_{\bH_{q,p}^{\phi,\gamma}(T)} \to 0
\end{equation*}
as $n, m\to \infty$. We call this sequence $u_n$ a defining sequence of $u$, and we define
\begin{equation*}
\partial_t^\alpha u= \lim_{n \to \infty} \partial_t^\alpha u_n \text{ in } \bH_{q,p}^{\phi,\gamma}(T).
\end{equation*}
The norm in $\bH_{q,p}^{\alpha,\phi,\gamma+2}(T)$ is naturally given by
\begin{equation*}
\|u\|_{{\bH_{q,p}^{\alpha,\phi,\gamma+2}}(T)}=\|u\|_{\bH_{q,p}^{\phi,\gamma+2}(T)}+\|\partial_t^\alpha u\|_{\bH_{q,p}^{\phi,\gamma}(T)}.
\end{equation*}

(ii) For $u\in{\bH_{q,p}^{\alpha,\phi,\gamma+2}(T)}$ and $u_0 \in B^{\phi,\gamma+2-\frac{2}{\alpha q}}_{p,q}$, we say  $u(0,x)=u_0$ if there exists a defining sequence $u_n$ such that 
$u_{n}(0,\cdot)\in C_{p}^{\infty}(\bR^d)$ and
\begin{equation*}
u_n(0,\cdot) \to u_0 \text{ in } B^{\phi,\gamma+2-\frac{2}{\alpha q}}_{p,q}.
\end{equation*}

(iii) We write $u\in{\bH_{q,p,0}^{\alpha,\phi,\gamma+2}(T)}$, if there exists a defining sequence $u_n$ such that
\begin{equation*}
u_n(0,x)=0 \qquad \forall x\in\bR^d, \quad \forall n\in\bN.
\end{equation*}

\end{definition}

\begin{remark} \label{Hvaluedconti}
(i) Obviously, $\bH_{q,p}^{\alpha,\phi,\gamma+2}(T)$ is a Banach space.

 (ii) Applying \eqref{eqn 07.19.15:09} to $u_n(t,x)-u_n(0,x)$ and using \eqref{eq:Lp continuity of I}, one can check that Definition \ref{defn defining} (ii) is independent of the choice of a defining sequence.
 
(iii) Actually, by Corollary \ref{cor 8.3.1},  $u\in{\bH_{q,p}^{\alpha,\phi,\gamma+2}(T)}$ and $u(0,\cdot)=0$ if and only if  $u\in{\bH_{q,p,0}^{\alpha,\phi,\gamma+2}(T)}$.

(iv) Following  \cite[Remark 3]{mikulevivcius2017p},  one can show that the embedding $H_p^{2n} \subset H_p^{\phi,2n}$ is continuous for any $n\in\bN$.  

 \end{remark}

\begin{lemma} \label{basicproperty}
Let $\alpha\in(0,1)$, $1<p,q<\infty$, $\gamma\in\R$, and $T<\infty$.

(i) The space $\bH_{q,p,0}^{\alpha,\phi,\gamma+2}(T)$ is a closed subspace of $\bH_{q,p}^{\alpha,\phi,\gamma+2}(T)$.

(ii) $C_c^\infty(\bR^{d+1}_+)$ is dense in $\bH_{q,p,0}^{\alpha,\phi,\gamma+2}(T)$.

(iii) For any $\gamma,\nu\in\bR$, $(1-\phi(\Delta))^{\nu/2}:\bH_{q,p}^{\alpha,\phi,\gamma+2}(T)\to\bH_{q,p}^{\alpha,\phi,\gamma-\nu+2}(T)$ is an isometry, and for any  $u\in\bH_{q,p}^{\alpha,\phi,\gamma+2}(T)$
\begin{equation*} \label{spacetimederiv change}
(1-\phi(\Delta))^{\nu/2}\partial_t^\alpha u=\partial_t^\alpha (1-\phi(\Delta))^{\nu/2} u.
\end{equation*}
\end{lemma}

\begin{proof}

(i) The proof is straightforward and is left to the reader.

(ii)  It suffices to show that for any given $u\in C^{\alpha,\infty}_{p}([0,T]\times\bR^d)$ with $u(0,\cdot)=0$, there exists a sequence $u_n\in C_c^\infty(\bR^{d+1}_+)$ such that
\begin{equation}
\label{eqn 6.2.1}
\|u_n-u\|_{\bH_{q,p}^{\alpha,\phi,\gamma+2}(T)}\to\infty
\end{equation}
as $n\to\infty$. By Remark \ref{Hvaluedconti} (iv) and considering a multiplication with smooth cut-off function of $x$, one may assume that $u$ has compact support,  that is, with some $R>0$, $u(t,x)=0$ whenever $|x|>R$.
Extend $u=0$ if $t>T$. Take a nonnegative smooth function $\eta_1\in C_c^\infty((1,2))$ so that
\begin{equation*}
\int_0^\infty \eta_1(t)dt=1.
\end{equation*}
For $\varepsilon >0$, we define
\begin{equation*}
\eta_{1,\varepsilon}(t)=\varepsilon^{-1}\eta_1(t/\varepsilon),
\end{equation*}
\begin{equation*}
u^{\varepsilon}(t,x)=\eta(t)\int_0^\infty u(s,x)\eta_{1,\varepsilon}(t-s)ds,
\end{equation*}
where $\eta\in C^\infty([0,\infty))$ such that $\eta(t)=1$ for all $t\leq T+1$ and vanishes for all large $t$. Then,  due to  $\eta_1\in C_c^\infty((1,2))$, 
\begin{align*}
u^{\varepsilon}(t,x)=0 \qquad \forall t<\varepsilon, \quad \forall x\in\bR^d,
\end{align*}
and $u^{\varepsilon_{1}}\in\Ccinf(\R^{d+1}_{+})$. Also using $u(0,x)=0$ and Fubini's theorem, one can prove 
\begin{equation*}
\partial_t^\alpha u^{\varepsilon}(t)=(\partial_t^\alpha u)^{\varepsilon}(t), \quad t\leq T.
\end{equation*}
Therefore, for any $n\in \bN$,
$$
\|u^{\varepsilon}-u\|_{L_q([0,T];H^{2n}_p)}+\|\partial_t^\alpha u^{\varepsilon}-\partial_t^\alpha u \|_{L_q([0,T];H^{2n}_p)} \to 0
$$
as $\varepsilon \downarrow 0$.  This and Remark \ref{Hvaluedconti} (iv) yields \eqref{eqn 6.2.1} with $u_n:=u^{1/n}$.
Therefore, (ii) is proved.

(iii) We first prove the claims  for functions  $u\in C^{\alpha,\infty}_{p}([0,T]\times\bR^d)$. Take $\eta\in C_c^\infty (\bR^d)$ such that $\eta=1$ on $B_1$ and $\eta=0$ outside of $B_2$. For $n\in\bN$, define $\eta_n(x)=\eta(x/n)$ and 
$$
 v_n:=\eta_n (1-\phi(\Delta))^{\nu/2} u\in C^{\alpha,\infty}_{p}([0,T]\times \bR^d).
$$
Then,  for any $m\in \bN$, $v_n \to (1-\phi(\Delta))^{\nu/2}u$ in $L_{q}((0,T);H_p^m)$ as $n\to \infty$, and therefore  by Remark \ref{Hvaluedconti} (iv), $v_{n} \to (1-\phi(\Delta))^{\nu/2} u$ in $\bH_{q,p}^{\phi,\gamma-\nu+2}(T)$ as $n\to \infty$.  Similarly,
$$
\partial^{\alpha}_tv_n=\eta_n (1-\phi(\Delta))^{\nu/2} \partial^{\alpha}_t u \to (1-\phi(\Delta))^{\nu/2} \partial^{\alpha}_t u
$$
in $\bH_{q,p}^{\phi,\gamma-\nu+2}(T)$ as $n\to \infty$. Therefore $u\in \bH_{q,p}^{\alpha,\phi,\gamma-\nu+2}(T)$, and all the assertions of (iii) also follow.

Now, let $u\in\bH^{\alpha,\phi,\gamma+2}_{q,p}(T)$. Take a  defining sequence $u_{n}$ for $u$.  Then, by the above result
\begin{equation*}
    \label{eqn 6.3.1}
(1-\phi(\Delta))^{\nu/2} u_n \in \bH_{q,p}^{\alpha,\phi,\gamma-\nu+2}(T), \quad \partial^{\alpha}_t (1-\phi(\Delta))^{\nu/2} u_n=(1-\phi(\Delta))^{\nu/2} \partial^{\alpha}_tu_n.
\end{equation*}
From these and the fact $u_n \to u$ in $\bH_{q,p}^{\alpha,\phi,\gamma+2}(T)$ as $n\to \infty$, it follows that   $(1-\phi(\Delta))^{\nu/2} u_n$ is a Cauchy sequence in 
$\bH_{q,p}^{\alpha,\phi,\gamma-\nu+2}(T)$. Let $w$ denote the limit in this space. Then, since $(1-\phi(\Delta))^{\nu/2} u_n \to (1-\phi(\Delta))^{\nu/2} u$ in 
$\bH^{\phi,\gamma-\nu+2}_p(T)$, we conclude $w=(1-\phi(\Delta))^{\nu/2} u\in  \bH_{q,p}^{\alpha,\phi,\gamma-\nu+2}(T)$. The claim for the isometry is obvious, and the other assertion of (iii) also follows. The
lemma is proved. 
\end{proof}

Here is the main result of this  article.

\begin{theorem} \label{main theorem}
Let $\gamma \in \bR, \alpha\in(0,1)$, $p,q\in(1,\infty)$. Suppose  Assumption \ref{ass bernstein} holds. Then for any $u_0\in B^{\phi,\gamma+2-2/\alpha q}_{p,q}$ and $f\in \bH_{q,p}^{\phi,\gamma}(T)$, the equation
\begin{equation}\label{mainequation}
\partial_t^\alpha u = \phi(\Delta)u + f,\quad t>0\,; \quad u(0,\cdot)=u_0
\end{equation}
admits a unique solution $u$ in the class $\bH_{q,p}^{\alpha,\phi,\gamma+2}(T)$, and   we have

\begin{equation} \label{mainestimate}
\|u\|_{\bH_{q,p}^{\alpha,\phi,\gamma+2}(T)}\leq N \left( \|f\|_{\bH_{q,p}^{\phi,\gamma}(T)}+ \|u_{0}\|_{B^{\phi,\gamma+2-2/\alpha q}_{p,q}} \right),
\end{equation}
where $N=N(\alpha,d,\phi,p,q,\gamma,T)$.  Moreover, if $u_0=0$ then
\begin{equation}
   \label{mainestimate-11}
\|\phi(\Delta)u\|_{\bH_{q,p}^{\phi,\gamma}(T)}\leq N_0 \|f\|_{\bH_{q,p}^{\phi,\gamma}(T)},
\end{equation}
where $N_0=N_0(\alpha,d,\phi,p,q,\gamma)$ is independent of $T$. 
\end{theorem}

\mysection{Estimates of the fundamental solution}

In this section we obtain  sharp bounds of arbitrary order derivatives of the fundamental solution  to equation \eqref{mainequation}. 

We first study  the derivatives of the transition density of $d$-dimensional subordinate Brownian motion.

\begin{lemma} \label{gammaversionj}
Let Assumption \ref{ass bernstein} hold. Then 
there exists a constant $N=N(c, \delta_0)$ such that
\begin{equation} \label{int phi}
\int_{\lambda^{-1}}^\infty r^{-1}\phi(r^{-2})dr \leq N \phi(\lambda^2), \quad \forall\, \lambda>0.
\end{equation}
\end{lemma}

\begin{proof}
Note first that Assumption \ref{ass bernstein} combined with the concavity of $\phi$ gives 
 \begin{equation}
\label{phiratio}
c\left(\frac{R}{r}\right)^{\delta_0}\leq\frac{\phi(R)}{\phi(r)}\leq\frac{R}{r}, \qquad 0<r<R<\infty.
\end{equation}

By the change of variables and \eqref{phiratio},
\begin{equation*}
\begin{aligned}
\int_{\lambda^{-1}}^{\infty}r^{-1}\phi(r^{-2})dr&=\int_{1}^{\infty}r^{-1}\phi(\lambda^{2}r^{-2})dr=\int_{1}^{\infty}r^{-1}\phi(\lambda^{2}r^{-2})\frac{\phi(\lambda^{2})}{\phi(\lambda^{2})}dr
\\
& \leq N \int_{1}^{\infty}r^{-1-2\delta_{0}}dr\phi(\lambda^{2})=N\phi(\lambda^{2}).
\end{aligned}
\end{equation*}

The lemma is proved.
\end{proof}

Recall that $S=\left(S_t\right)_{t\ge0}$ is a subordinator with Laplace exponent $\phi$ and  $W=\left(W_t\right)_{t\ge0}$ is a Brownian motion in $\bR^d$, independent of $S$. We call $X_t:=W_{S_t}$  the subordinate Brownian motion. It is known (see e.g. \cite{bogdan2009potential,sato1999}) that $X_{t}$ is rotationally invariant L\'evy process in $\bR^d$ with chracteristic exponent $\phi(|\xi|^2)$. That is,
\begin{equation} \label{fourier def sub}
\bE \left[e^{i\xi\cdot X_t }\right]=e^{-t\phi(|\xi|^2)}, \qquad \forall\,\, \xi\in\bR^d, \, t>0.
\end{equation}
Using \eqref{fourier def sub}  and the equality  
$$e^{-|z|^2}=(4\pi)^{-d/2} \int_{\R^d}e^{i\xi \cdot z}e^{-|\xi|^2/4}d \xi,$$
 we have for $\lambda>0$
\begin{align}
&\bE[e^{-\lambda |X_t|^2}]=(4\pi)^{-d/2} \int_{\R^d}\bE[ e^{i \sqrt \lambda \xi \cdot X_t}]e^{-|\xi|^2/4}d \xi \nonumber\\
&=(4\pi)^{-d/2} \int_{\R^d} e^{-t\phi( \lambda|\xi|^2)} e^{-|\xi|^2/4}d \xi. \nonumber
\end{align}
Therefore, for $t,\lambda>0$,
\begin{eqnarray}
0&\leq&  \bE  [e^{-\lambda|X_t|^2} - e^{-2\lambda|X_t|^2}] \nonumber \\
&=&(4\pi)^{-d/2} \int_{\R^d} (e^{-t\phi(  \lambda|\xi|^2)}-e^{-t\phi( {2\lambda}|\xi|^2)}) e^{-|\xi|^2/4}d\xi   \nonumber  \\
&=&N(d)\int_0^\infty (e^{-t\phi(  \lambda r^2)}-e^{-t\phi({2 \lambda} r^2)}) e^{-r^2/4} r^{d-1} dr \nonumber\\
&=:&N(d)g_t(\lambda).  \label{e:PP0}
\end{eqnarray}

The following lemma is a version of \cite[Lemma 2.1]{kim2013parabolic}.

\begin{lemma}\label{l:pub1}
Let Assumption \ref{ass bernstein} hold.  Then, there exists a constant $N=N(c,\delta_0,d)$ such that for every $t, v  >0$
$$
g_t( v^{-1}) \le  N t \phi( v^{-1}).
$$
\end{lemma}

\begin{proof}
Note that by \eqref{phiratio}  for any $r>0$  
\begin{align*}
&\frac{1}{t \phi( v^{-1})}
\le
\frac{\phi( {2} v^{-1} r^2 )+\phi( v^{-1} r^2)}{\phi( v^{-1})}\frac{1}{t|\phi( {2} v^{-1} r^2 )- \phi( v^{-1} r^2)|}\\
\leq&
N(r^2\vee r^{2\delta_0}) \frac{1}{t|\phi(  {2} v^{-1} r^2 )- \phi( v^{-1} r^2)|}.
\end{align*}
Thus using the inequality $|e^{-a}-e^{-b}| \le |a-b|$, $a, b >0$, we have
\begin{align*}
&\frac{g_t( v^{-1})}{t \phi( v^{-1/2})}\\
\leq& \ 
N\int_0^\infty \frac{|e^{-t\phi( v^{-1} r^2)}-e^{-t\phi(  {2} v^{-1} r^2 )}|}{t|\phi(  2v^{-1} r^2 )- \phi( v^{-1} r^2)|}  e^{-r^2/4} r^{d-1}(r^2\vee r^{2\delta_0})  dr\\
\leq& \ 
N\int_0^\infty  e^{-r^2/4} r^{d-1}(r^2\vee r^{2\delta_0})  dr < \infty.
\end{align*}
Therefore the lemma is proved.
\end{proof}

Let $p(t,x)=p_d(t,x)$ be the transition density of $X_t=W_{S_t}$, the $d$-dimensional subordinate Brownian motion corresponding to $\phi$. Then it is known (\cite[Section 5.3.1]{bogdan2009potential}) that for any $t>0, x\in \bR^d$,
\begin{eqnarray}
\label{eqn 7.20.1}
p(t,x)=p_d(t,x)&=&\frac{1}{(2\pi)^d} \int_{\bR^d} e^{i \xi \cdot x} e^{-t\phi(|\xi|^2)} \,d\xi\\
&=&\int_{(0,\infty)} (4\pi s)^{-d/2} \exp \left(-\frac{|x|^2}{4s}\right) \eta_t (ds),  \label{eqn 7.20.2}
\end{eqnarray}
where $\eta_t(ds)$ is the distribution of $S_t$.  Thus $X_t$  is rotationally invariant.  We put $p(t,r):=p(t,x)$  if $r=|x|$ for notational convenience.  It follows from \eqref{eqn 7.20.2} that $r\to p(t,r)$ is a decreasing function.

\begin{lemma} \label{t:hku}
Let Assumption \ref{ass bernstein} hold. Then there exists a constant $N=N(c, \delta_0, d)$ such that
 $$
 p_d(t, r) \leq N t  r^{-d} \phi(r^{-2}), \quad \quad (t, r) \in (0, \infty) \times (0,\infty).
$$
\end{lemma}
\begin{proof}
Fix $t \in (0, \infty)$. For $r \geq 0$ define $f_t(r)=r^{d/2}p(t, r^{1/2}).$
Since $r\to p(t,r)$ is decreasing, for $r \ge 0$,
\begin{align}
&\bP(  \sqrt{r/2} <|X_t| <\sqrt{r}) =\int_{\sqrt{r/2} <|y| <\sqrt{r}} p(t, |y|)dy \nonumber\\
&\ge |B_1(0)| (1-2^{-d/2})  r^{d/2} p(t, r^{1/2})=  |B_1(0)| (1-2^{-d/2})  f_t(r). \nonumber
\end{align}
Denoting $\cL f_t(\lambda)$ the Laplace transform of $f_t$, we have
\begin{align}
&\cL f_t(\lambda) \leq N\int_0^\infty \bP(  \sqrt{r/2} <|X_t| <\sqrt{r}) e^{-\lambda r}dr=
N\bE \int_{|X_t|^2}^{2|X_t|^2}e^{-\lambda r}dr \nonumber
\\
&=N\lambda^{-1} \bE [e^{-\lambda|X_t|^2} - e^{-2\lambda|X_t|^2}]=N\lambda^{-1} g_t(\lambda), \quad  \lambda > 0
\label{e:PP6}
\end{align}
from \eqref{e:PP0}.  Furthermore, for any $v>0$
\begin{align*}
\cL f_t(v^{-1})=&\int^\infty_0e^{-av^{-1} } f_t(a)\, da= v \int^\infty_0e^{-s}f_t\left(sv\right)ds\\
\ge& \ v \int^{1}_{1/2} e^{-s}f_t\left(sv\right)ds= v \int^{1}_{1/2} e^{-s} s^{d/2}v^{d/2} p\left(t,s^{1/2}v^{1/2}\right)ds \\
\ge& \  v 2^{-d/2}  v^{d/2}  p\left(t,v^{1/2}\right) \int^{1}_{1/2} e^{-s}ds = 2^{-d/2} v f_t\left(v \right)\left(\int^{1}_{1/2} e^{-s}ds\right).
\end{align*}
Thus
\begin{align}
f_t\left(v\right)\leq 2^{d/2}\frac{v^{-1}\cL f_t(v^{-1})}{e^{-1/2}-e^{-1}}.
\label{e:PP7}
\end{align}
 Finally, combining  \eqref{e:PP6} and \eqref{e:PP7} with Lemma \ref{l:pub1} we
 conclude
 $$
 p(t, r) = r^{-d} f_t(r^2) \le N r^{-d-2}\cL f_t(r^{-2}) \le N r^{-d}g_t(r^{-2}) \le N  t  r^{-d} \phi( r^{-2}).
$$
The lemma is proved.
\end{proof}

   The difference of the following result  from those in the literature is that it 
only concerns the upper bound rather than two-side estimate, but it holds under relatively weaker assumption and gives estimate which holds for all $t>0$.   See Remark \ref{remark heat} for some related results.
 
\begin{lemma} \label{pestimate}
Let Assumption \ref{ass bernstein} hold. Then, there exists a constant $N=N(d,c,\delta_0)$ such that for $(t,x)\in(0,\infty)\times \bR^d$,
\begin{equation} \label{pestimate ineq}
p_d(t,x)\leq N\left(\left(\phi^{-1}(t^{-1})\right)^{d/2}\wedge t\frac{\phi(|x|^{-2})}{|x|^d}\right).
\end{equation}
\end{lemma}
\begin{proof}
By Lemma \ref{t:hku}, we only need to prove
$$
p(t,x)\leq N\left(\phi^{-1}(t^{-1})\right)^{d/2}.
$$
We modify the proof of \cite[Corollary 3.5]{kim2013parabolic}.
Note that
$$
t\phi(|\xi|^2)=\frac{\phi(|\xi|^2)}{\phi(\phi^{-1}(t^{-1}))}.
$$
If $|\xi|^2>\phi^{-1}(t^{-1})$, then by \eqref{phiratio},
\begin{equation} \label{eq 200615}
\frac{\phi(|\xi|^2)}{\phi(\phi^{-1}(t^{-1}))} \geq c\left(\frac{|\xi|^2}{\phi^{-1}(t^{-1})}\right)^{\delta_0}.
\end{equation}
\\
Using \eqref{eqn 7.20.1} and \eqref{eq 200615},
$$
\begin{aligned}
p(t,x)&=\frac{1}{(2\pi)^d}\int_{\bR^d} e^{i\xi\cdot x} e^{-t\phi(|\xi|^2)}d\xi
\\
&\leq \left|\frac{1}{(2\pi)^d}\int_{|\xi|^2>\phi^{-1}(t^{-1})} e^{i\xi\cdot x} e^{-t\phi(|\xi|^2)}d\xi\right| 
\\
&\quad+ \left|\frac{1}{(2\pi)^d}\int_{|\xi|^2\leq\phi^{-1}(t^{-1})} e^{i\xi\cdot x} e^{-t\phi(|\xi|^2)}d\xi \right|
\\
&\leq \frac{1}{(2\pi)^d} \int_{|\xi|^2>\phi^{-1}(t^{-1})} e^{-c\left(\frac{|\xi|^2}{\phi^{-1}(t^{-1})}\right)^{\delta_0}} d\xi + \frac{1}{(2\pi)^d} \int_{|\xi|^2\leq\phi^{-1}(t^{-1})} d\xi
\\
&=N \left(\phi^{-1}(t^{-1})\right)^{d/2}\int_{|\xi|^2>1} e^{-c|\xi|^{2\delta_0}} d\xi + N\left(\phi^{-1}(t^{-1})\right)^{d/2}
\\
&\leq N\left(\phi^{-1}(t^{-1})\right)^{d/2}.
\end{aligned}
$$
The lemma is proved.
\end{proof}

\begin{remark}
\label{remark heat}
(i) Inequality \eqref{pestimate ineq}  was introduced in \cite[Corollary 3.5]{kim2013parabolic}  on  finite time interval $[0,T]$ with a constant $N$ depending also on $T$. Condition (H1) is used in \cite[Corollary 3.5]{kim2013parabolic}.

(ii) Assume that there exist constants $0<\delta_1\leq \delta_2<1$ and $c_1,c_2>0$ such that
\begin{equation}
\label{eqn 7.21.1}
c_1\left(\frac{R}{r}\right)^{\delta_1}\leq\frac{\phi(R)}{\phi(r)}\leq c_2\left(\frac{R}{r}\right)^{\delta_2}, \qquad 0<r<R<\infty.
\end{equation}
 Then, by \cite[Section 3]{kim2014global}, we have the sharp two-sided estimate
\begin{align} \label{special sharp}
p(t,x) \sim \left(\left(\phi^{-1}(t^{-1})\right)^{d/2}\wedge t\frac{\phi(|x|^{-2})}{|x|^d}\right).
\end{align} 
Note that  \eqref{eqn 7.21.1}  is stronger than  Assumption \ref{ass bernstein}. It turns out that  Assumption \ref{ass bernstein} is insufficient  for  equivalence relation \eqref{special sharp}. See  \cite[Theorem 4.1]{chen2012sharp}.
\end{remark}

The following result is a consequence of \eqref{eqn 7.20.2} and Lemma \ref{pestimate}.

\begin{lemma} \label{pderivativeestimate}
Let Assumption \ref{ass bernstein} hold. Then there exists a constant $N=N(d,c,\delta_0,m)$ so that for any $(t,x)\in (0,\infty)\times \bR^d$,
\begin{align*}
& |D^m_x p_d(t,x)| \leq N \sum_{m-2n\geq0,n\in\bN_0} |x|^{m-2n}\left((\phi^{-1}(t^{-1}))^{d/2+m-n}\wedge t\frac{\phi(|x|^{-2})}{|x|^{d+2(m-n)}}\right).
\end{align*}

\begin{proof}
For $n\in \bN$, let $0_{2n}$ denote the origin in $\bR^{2n}$. Thus, for $x\in \bR^d$ we have $(x,0_{2n})\in \bR^{d+2n}$. 
By the dominated convergence theorem,
\begin{align*}
\partial_{x^i} p_d(t,x)=& \int_{(0,\infty)}(4\pi s)^{-d/2} \partial_{x^i}\exp\left(-\frac{|x|^2}{4s}\right)\eta_t(ds)
\\
=& -\frac{x^i}{2} \int_{(0,\infty)} s^{-1}(4\pi s)^{-d/2} \exp\left(-\frac{|x|^2}{4s}\right)\eta_t(ds)
\\
=& -2\pi x^i p_{d+2}(t,(x,0_2)).
\end{align*}
The last equality is due to  \eqref{eqn 7.20.2}. 

Similarly,
\begin{align*}
\partial_{x^i x^i} p_d(t,x)=4\pi^2 (x^i)^2 p_{d+4}(t,(x,0_4))-2\pi p_{d+2}(t,(x,0_2)),
\end{align*}
and, for $i\neq j$,
$$
\partial_{x^i x^j} p_d(t,x)=4\pi^2 x^ix^j p_{d+4}(t,(x,0_4)).
$$
Repeating the product rule of differentiation (or using the induction argument), one can check that $D^m_xp_d(t,x)$ is a sum of functions of the type 
$$
x^{\sigma} p_{d+2(m-n)}(t,(x,0_{2(m-n)})), \quad \quad  |\sigma|=m-2n\geq 0.
$$
Thus we get the claim of the lemma by applying \eqref{pestimate ineq} for dimensions $d+2,\dots,d+2m$.
\end{proof}

\end{lemma}

Next we study  the fundamental solution $q(t,x)$ to the equation
\begin{equation}
        \label{eqn 7.20.5}
\partial_{t}^{\alpha}u(t,x)=\phi(\Delta) u(t,x), \quad t>0\, ;\, u(0,x)=u_0.
\end{equation}
That is, $q(t,x)$ is the function such that under appropriate smoothness condition on $u_0$, the function $u(t,x):=(q(t,\cdot)\ast u_0(\cdot))(x)$ satisfies equation \eqref{eqn 7.20.5}.

\vspace{1mm}

Recall that $X_t$ is the $d$-dimensional subordinate Brownian motion with transition density  $p(t,x)$. Let $Q_t$ be an increasing L\'evy process independent of $X_t$ having the Laplace transform
\begin{equation*}
\bE \exp(-\lambda Q_t)= \exp(-t\lambda^\alpha).
\end{equation*}
Let 
$$
R_t:=\inf\{s>0 : Q_s>t\}
$$ be the inverse process of the  subordinator $Q_{t}$, and let  $\varphi(t,r)$ denote the probability density function of $R_t$. Then,  as is shown in   Lemma \ref{zero converge} (cf. \cite[Theorem 1.1]{chen2017time}), the function
\begin{eqnarray}
q(t,x):=\int_0^\infty p(r,x)d_r \bP(R_t\leq r)
=\int_0^\infty p(r,x) \varphi(t,r) \,dr   \label{eqn 8.1.1}
\end{eqnarray}
becomes the fundamental solution to  equation \eqref{eqn 7.20.5}.   Actually the definition of  $q(t,x)$ implies that $q(t,x)$ is the transition density of $Y_t:=X_{R_t}$, which is called subordinate Brownian motion delayed by an inverse subordinator.

Let $E_{\alpha,\beta}$ be the two-parameter Mittag-Leffler function defined as
\begin{equation*}
E_{\alpha,\beta}(z)=\sum_{k=0}^\infty \frac{z^k}{\Gamma(\alpha k+\beta)}, \quad z\in\bC, \alpha>0, \beta\in\bC,
\end{equation*}
and $E_{\alpha}:=E_{\alpha,1}$. For later use, we  note the recurrence relation which follows immediately from the definition:
\begin{equation} \label{Mittag recurr}
E_{\alpha,\beta}(z)=\frac{1}{\Gamma(\beta)} + zE_{\alpha,\beta+\alpha}(z),
\end{equation}
where $ \frac{1}{\Gamma(\beta)}:=0$ if $\beta=0,-1,-2,\cdots$.

For $\beta\in \bR$, denote
$$
 \varphi_{\alpha,\beta}(t,r):=D_t^{\beta-\alpha}\varphi(t,r):=(D^{\beta-\alpha}_t \varphi(\cdot,r))(t),
 $$
and for  $(t,x)\in(0,\infty)\times \bR^{d}\setminus\{0\}$ define
\begin{equation*}
q_{\alpha,\beta}(t,x):=\int_0^\infty p(r,x)\varphi_{\alpha,\beta}(t,r) dr.
\end{equation*}

 \begin{lemma}
 \label{lem 7.21.1}
 Let $\alpha\in (0,1)$ and $\beta\in \bR$.
 
  (i) $\varphi_{\alpha,\beta}(t,r)$ has the series repesentation
  \begin{eqnarray}
\varphi_{\alpha,\beta}(t,r)=&t^{-\beta}\sum_{k=0}^\infty \frac{(-rt^{-\alpha})^k}{k!\Gamma(1-\beta-\alpha k)}=t^{-\beta} {_0\Phi_1}\left[\begin{array}{cc}
 - \\
(1-\beta,-\alpha)
\end{array}\Bigg| -rt^{-\alpha}\right],   \label{eqn 7.21.5}
\end{eqnarray}
where ${_0\Phi_1}\left[\begin{array}{cc}
 - \\
(a,b)
\end{array}\Bigg| z\right]:=\sum_{k=0}^{\infty}  \frac{z^k}{k!\Gamma(a+bk)}$ is the Wright function.

(ii) There exist  constants $c,N>0$ depending only on $\alpha,\beta$ such that 
\begin{equation} \label{philarge}
|\varphi_{\alpha,\beta}(t,r)|\leq N t^{-\beta}e^{-c(rt^{-\alpha})^{1/(1-\alpha)}}
\end{equation}
for $rt^{-\alpha}\geq 1$, and 
\begin{eqnarray}\label{betainteger}
|\varphi_{\alpha,\beta}(t,r)|\leq \left\{ \begin{array}{ll}  N rt^{-\alpha-\beta}~&\beta\in \bN \\ Nt^{-\beta}~&\beta \notin \bN\end{array} \right.
\end{eqnarray}
for $rt^{-\alpha}\leq 1$.  

(iii) For  any $(t,x)\in(0,\infty)\times \bR^{d}\setminus\{0\}$, 
\begin{equation} \label{eqn 07.28 20:25}
D_t^{\beta-\alpha}q(t,x)=q_{\alpha,\beta}(t,x).
\end{equation}

(iv) For any $t>0, \xi\in \bR^d$,
\begin{equation} \label{fourier of q}
\hat{q}_{\alpha,\beta}(t,\xi) =t^{\alpha-\beta} E_{\alpha,1-\beta+\alpha}(-t^\alpha\phi(|\xi|^2)).
\end{equation}

 \end{lemma}
 
 \begin{proof}
(i)  By \cite[Proposition 1(a)]{bingham1971} (or \cite[Theorem 4.3]{bondesson1996}, \cite[Remark 3.1]{meerschaert2004limit}) 
 $$
 \bE e^{-s R_t  }=\int_{(0,\infty)}e^{-s r} \varphi(t,r)dr=\sum_{k=0} \frac{(-st^{\alpha})^k}{\Gamma(\alpha k+1)}=E_{\alpha}(-st^{\alpha}).
 $$
 Formula (3.16) in  \cite{kochubei2012fractional}  and the equality below (3.17) in  \cite{kochubei2012fractional} lead to the following   Wright function representation
\begin{equation*}
\varphi(t,r)=t^{-\alpha} {_0\Phi_1}\left[\begin{array}{cc}
 - \\
(1-\alpha,-\alpha)
\end{array}\Bigg| -rt^{-\alpha}\right] := t^{-\alpha}\sum_{k=0}^\infty \frac{(-rt^{-\alpha})^k}{k!\Gamma(1-\alpha-\alpha k)}.
\end{equation*}
Therefore, \eqref{eqn 7.21.5}   follows from 
 \cite[Theorem 3.1]{kilbas2013multi} if $\beta\leq \alpha$ and  from  \cite[Theorem 3.5]{kilbas2013multi}  if $\beta\geq \alpha$.

\vspace{1mm}

(ii) If $rt^{-\alpha}\geq 1$, then \eqref{philarge} is a consequence of the asymptotic behavior of the Wright function (see e.g. \cite[Theorem 1]{wright1940generalized} or \cite[Theorem 25]{braaksma1936asymptotic}).

Also, if $rt^{-\alpha}\leq 1$,  then the series representation of $\varphi_{\alpha,\beta}(t,r)$ easily yields \eqref{betainteger}.  

\vspace{1mm}

(iii)  By Lemma \ref{pestimate},
\begin{equation}
	\label{eqn 07.28.22:31}
 |p(r,x)| \leq N(d,c,\delta_0,x) r.
\end{equation}
This,  \eqref{philarge} and \eqref{betainteger} easily yield 
\begin{equation} \label{q well def}
\int_0^\infty |p(r,x) \varphi_{\alpha,\beta}(t,r)| dr \leq N \int_0^\infty r \, |\varphi_{\alpha,\beta}(t,r)| dr< \infty.
\end{equation}

Now we prove \eqref{eqn 07.28 20:25}. First assume $\beta < \alpha$. Then, since $\varphi(t,\cdot)\geq 0$,  \eqref{q well def}  implies
$$
\int_0^\infty \int_0^t (t-s)^{\alpha-\beta-1} \varphi(s,r) |p(r,x)| ds dr <\infty,
$$
and therefore  Fubini's theorem yields the desired result.

Next we assume $\beta > \alpha$. Take $n\in\bN$ such that $n-1\leq \beta-\alpha<n$. By the above result, we have
$$
I^{n-(\beta-\alpha)}_{t} q (t,x) =D^{(\beta-n)-\alpha}_{t}q(t,x)= q_{\alpha,\beta-n} (t,x).
$$
Hence, by the definition of fractional derivative, it remains to prove that
\begin{equation*}
	\label{eqn 07.28.22:55}
\left(\frac{d}{dt}\right)^n q_{\alpha,\beta-n} (t,x)=q_{\alpha,\beta}(t,x).
\end{equation*}
For this, due to  the dominated convergence theorem, we only need to show for any $0<\varepsilon<T$  and $k=0,1,2,\dots,n$,
\begin{equation}
	\label{eqn 07.28.23:04}
\int_0^\infty \left|p(r,x)\right| \times \left( \sup_{t\in[\varepsilon,T]} |\varphi_{\alpha,\beta-k}(t,r)| \right) dr <\infty.
\end{equation}
For $t\in[\varepsilon,T]$, by \eqref{philarge} and \eqref{betainteger}, there exist $c,N>0$ such that 
 $$
|\varphi_{\alpha,\beta-k}(t,r)|\leq Ne^{-cr^{1/(1-\alpha)}} \quad \text{if} \quad r\geq \varepsilon^\alpha,
 $$
and
 $$
|\varphi_{\alpha,\beta-k}(t,r)|\leq N \quad \text{if} \quad r\leq T^\alpha,
 $$
 where constants $c,N$ depend on $\alpha,\beta,k,\varepsilon$ and $T$. Combining this with \eqref{eqn 07.28.22:31} we have \eqref{eqn 07.28.23:04}. The claim of (iii) is proved.

(iv) By  (34) of \cite{gorenflo2007analytical}, we have
\begin{align} \label{Laplace trans}
\int_0^\infty e^{-sr}\varphi_{\alpha,\beta}(t,r)dr = t^{\alpha-\beta} E_{\alpha,1-\beta+\alpha}(-st^\alpha).
\end{align}
Therefore, by \eqref{Laplace trans} and Fubini's theorem 
\begin{align*} 
\hat{q}_{\alpha,\beta}(t,\xi) &= \int_0^\infty \varphi_{\alpha,\beta}(t,r)\left[\int_{\bR^d}e^{-ix\cdot \xi} p(r,x)dx \right] dr \nonumber
\\
&=\int_0^\infty \varphi_{\alpha,\beta}(t,r)e^{-r\phi(|\xi|^2)} dr \nonumber
\\
&=t^{\alpha-\beta} E_{\alpha,1-\beta+\alpha}(-t^\alpha\phi(|\xi|^2)).
\end{align*}
The lemma is proved.
\end{proof}

\begin{lemma} \label{qestimatebeta123}
Let $\alpha\in(0,1)$ $\beta\in\bR$, and $m\in\bN_{0}$.  Then
there exists a constant $N=N(\alpha,\beta, d, c, \delta_0 , m)$ such that 

(i)  for any $t>0, x\in \bR^d\setminus\{0\}$
\begin{equation}
\label{eqn 7.21.7}
 |D^m_x q_{\alpha,\beta} (t,x)|\leq Nt^{2\alpha-\beta}\frac{\phi(|x|^{-2})}{|x|^{d+m}}, 
\end{equation}

(ii) furthermore,  if $t^\alpha \phi(|x|^{-2})\geq1$ then 
\begin{align*}
&|D^m_x q_{\alpha,\beta} (t,x)| \leq N \int_{(\phi(|x|^{-2}))^{-1}}^{2t^{\alpha}} (\phi^{-1}(r^{-1}))^{(d+m)/2} rt^{-\alpha-\beta} dr, \qquad \beta\in\bN,
\end{align*}
and
\begin{align*}
&|D^m_x q_{\alpha,\beta} (t,x)| \leq N \int_{(\phi(|x|^{-2}))^{-1}}^{2t^{\alpha}} (\phi^{-1}(r^{-1}))^{(d+m)/2} t^{-\beta} dr, \qquad \beta\notin\bN.
\end{align*}
\end{lemma}

\begin{proof}
Take $x\in \bR^d\setminus\{0\}$. Then for any $r>0$ and  $y\neq 0$ sufficiently close to $x$, we have
$$
 |D^{\sigma}p(r,y)| \leq N(\phi,x,d,m) r, \quad |\sigma|\leq m
$$
due to Lemma \ref{pderivativeestimate}. Using \eqref{q well def} and the dominated convergence theorem, we get
\begin{equation}
\label{eqn 7.20.6}
D^m_x q_{\alpha,\beta} (t,x)=\int_0^\infty D^m_x p(r,x)\varphi_{\alpha,\beta}(t,r) dr.
\end{equation}

For $n\leq m/2$, set
\begin{equation}
   \label{eqn 727.1}
P_{m,n}(r,x) := |x|^{m-2n}\left((\phi^{-1}(r^{-1}))^{d/2+m-n}\wedge r\frac{\phi(|x|^{-2})}{|x|^{d+2(m-n)}}\right).
\end{equation}
By Lemma \ref{pderivativeestimate} and \eqref{eqn 7.20.6}, to estimate $|D^m_x q_{\alpha,\beta} (t,x)|$,  it is enough to control
\begin{align*}
\int_0^{\infty} |P_{m,n}(r,x)\varphi_{\alpha,\beta}(t,r)| dr=I+II,
\end{align*}
where
\begin{align}
I:=\int_0^{t^\alpha} |P_{m,n}(r,x)\varphi_{\alpha,\beta}(t,r)| dr,\quad II:=\int_{t^\alpha}^\infty |P_{m,n}(r,x)\varphi_{\alpha,\beta}(t,r)| dr. \label{eqn 7.27.4}
\end{align}

(i) We  prove \eqref{eqn 7.21.7} following  the argument in the proof of \cite[Theorem 1.4]{chen2020time}. 
Note that $|\varphi_{\alpha,\beta}(t,r)|\leq N t^{-\beta}$ for $r\leq t^\alpha$ due to \eqref{betainteger}.  Thus,
\begin{align*}
I&\leq N\int_0^{t^\alpha} rt^{-\beta} \frac{\phi(|x|^{-2})}{|x|^{d+m}}dr
\\
&\leq N t^{2\alpha-\beta}\frac{\phi(|x|^{-2})}{|x|^{d+m}}.
\end{align*}
For $II$, by \eqref{philarge} and the change of variables,
\begin{align*}
II&\leq N t^{-\beta} \int_{t^\alpha}^\infty  r\frac{\phi(|x|^{-2})}{|x|^{d+m}}e^{-c(rt^{-\alpha})^{1/(1-\alpha)}} dr
\\
&= N t^{2\alpha-\beta} \frac{\phi(|x|^{-2})}{|x|^{d+m}} \int_1^\infty re^{-cr^{1/(1-\alpha)}}dr
\\
&\leq N t^{2\alpha-\beta}\frac{\phi(|x|^{-2})}{|x|^{d+m}}.
\end{align*}
Hence, \eqref{eqn 7.21.7} is proved. 

(ii) We follow the argument in the proof of \cite[Theorem 4.8]{chen2020time}.

As in the proof of  (i), we estimate  $I$ and $II$ which are defined in \eqref{eqn 7.27.4}. 
 First, assume  $\beta\in\bN$. From \eqref{betainteger} we have
$$
I\leq N \int_0^{t^\alpha} |P_{m,n}(r,x)|rt^{-\alpha-\beta} dr=:I_1+I_2,$$
where
\begin{align*}
I_{1}:=\int_0^{(\phi(|x|^{-2}))^{-1}}|P_{m,n}(r,x)|rt^{-\alpha-\beta}dr,\quad I_{2}:=\int_{(\phi(|x|^{-2}))^{-1}}^{t^\alpha} |P_{m,n}(r,x)|rt^{-\alpha-\beta}dr.
\end{align*}
By \eqref{eqn 727.1},
\begin{align} \label{I_1 estimate 1}
I_1\leq N \int_0^{(\phi(|x|^{-2}))^{-1}} r^2t^{-\alpha-\beta} \frac{\phi(|x|^{-2})}{|x|^{d+m}}dr
\leq N t^{-\alpha-\beta}\frac{(\phi(|x|^{-2}))^{-2}}{|x|^{d+m}}.
\end{align}
Note that if 
$$
\left(\phi(|x|^{-2})\right)^{-1} \leq r\leq 2\left(\phi(|x|^{-2})\right)^{-1},
$$
then
$$
\phi^{-1}(r^{-1}) \leq |x|^{-2}\leq  \phi^{-1}(2r^{-1}).
$$
Applying \eqref{phiratio} with $R=\phi^{-1}(2r^{-1})$ and $r=\phi^{-1}(r^{-1})$ we have
\begin{align} \label{I_1 estimate 2}
\phi^{-1}(r^{-1}) \leq |x|^{-2}\leq  N(c, \delta_0) \phi^{-1}(r^{-1}).
\end{align}
Therefore, by \eqref{I_1 estimate 1} and \eqref{I_1 estimate 2}
\begin{align}
	\label{eqn 07.27.22.46} 
I_{1}\leq N t^{-\alpha-\beta}\frac{(\phi(|x|^{-2}))^{-2}}{|x|^{d+m}} &\leq N \int_{(\phi(|x|^{-2}))^{-1}}^{2(\phi(|x|^{-2}))^{-1}} (\phi^{-1}(r^{-1}))^{(d+m)/2} rt^{-\alpha-\beta} dr \nonumber
\\
&\leq N \int_{(\phi(|x|^{-2}))^{-1}}^{2t^{\alpha}} (\phi^{-1}(r^{-1}))^{(d+m)/2} rt^{-\alpha-\beta} dr.
\end{align}

To estimate $I_2$, we note that since $\phi$ and $\phi^{-1}$ are increasing, $r\phi(|x|^{-2})\geq1$ implies
\begin{equation} \label{wedgepoint}
|x|^{2} \leq \left(\phi^{-1}(r^{-1})\right)^{-1}.
\end{equation}
Applying \eqref{wedgepoint} and \eqref{eqn 727.1} again,
\begin{align}
	\label{eqn 07.27.23:07} 
I_2 &\leq N |x|^{m-2n}\int_{(\phi(|x|^{-2}))^{-1}}^{2t^\alpha} (\phi^{-1}(r^{-1}))^{d/2+m-n} rt^{-\alpha-\beta} dr \nonumber
\\
&\leq N \int_{(\phi(|x|^{-2}))^{-1}}^{2t^{\alpha}} (\phi^{-1}(r^{-1}))^{(d+m)/2} rt^{-\alpha-\beta} dr.
\end{align}
Now we estimate $II$. By (\ref{phiratio}) with $R=\phi^{-1}(t^{-\alpha})$ and $r=\phi^{-1}(r^{-1})$ we find that
\begin{align*}
\phi^{-1}(r^{-1})\leq t^\alpha r^{-1} \phi^{-1}(t^{-\alpha})\quad \text{if} \quad t^\alpha \leq r.
\end{align*}
Therefore, by the change of variables $r \to t^{\alpha}r$,
\begin{align} \label{I_2 estimate 1}
II &\leq N |x|^{m-2n}\int_{t^\alpha}^\infty (\phi^{-1}(r^{-1}))^{d/2+m-n} t^{-\beta}e^{-c(rt^{-\alpha})^{1/(1-\alpha)}} dr \nonumber
\\
&\leq N |x|^{m-2n} \int_{t^\alpha}^\infty \left(t^\alpha r^{-1} \phi^{-1}(t^{-\alpha})\right)^{d/2+m-n} t^{-\beta}e^{-c(rt^{-\alpha})^{1/(1-\alpha)}} dr \nonumber
\\
&= N |x|^{m-2n} t^{\alpha-\beta} (\phi^{-1}(t^{-\alpha}))^{d/2+m-n} \int_1^\infty  r^{-d/2-m+n} e^{cr^{1/(1-\alpha)}} dr \nonumber
\\
&\leq N |x|^{m-2n} t^{\alpha-\beta} (\phi^{-1}(t^{-\alpha}))^{d/2+m-n} \nonumber
\\
&\leq N t^{\alpha-\beta} (\phi^{-1}(t^{-\alpha}))^{(d+m)/2}.
\end{align}
For the last inequality above, we used $t^\alpha \phi(|x|^{-2})\geq1$. 

Note that if $t^{\alpha} \leq r\leq 2t^{\alpha}$, then
\begin{equation} \label{I_2 estimate 2}
\phi^{-1}(t^{-\alpha})\leq \phi^{-1}(2r^{-1})\leq N(c, \delta_0) \phi^{-1}(r^{-1}),
\end{equation}
where the last inequality is from \eqref{phiratio} with $R=\phi^{-1}(2r^{-1})$ and $r=\phi^{-1}(r^{-1})$. Therefore, we have 
\begin{align*}
t^{\alpha-\beta} (\phi^{-1}(t^{-\alpha}))^{(d+m)/2} &\leq N \int_{t^\alpha}^{2t^{\alpha}} (\phi^{-1}(r^{-1}))^{(d+m)/2} rt^{-\alpha-\beta} dr
\\
&\leq N \int_{(\phi(|x|^{-2}))^{-1}}^{2t^{\alpha}} (\phi^{-1}(r^{-1}))^{(d+m)/2} rt^{-\alpha-\beta} dr,
\end{align*}
and this gives the desired result in (ii) for $\beta\in\bN$.

Next, we assume  $\beta\notin\bN$. We repeat the above argument used to estimate $I$ and $II$.  For $I$, by \eqref{betainteger}, we need to estimate
\begin{align*}
\tilde{I}_{1}:=\int_0^{(\phi(|x|^{-2}))^{-1}}|P_{m,n}(r,x)|t^{-\beta}dr,\quad \tilde{I}_{2}:=\int_{(\phi(|x|^{-2}))^{-1}}^{t^\alpha} |P_{m,n}(r,x)|t^{-\beta}dr,
\end{align*}
instead of $I_1$ and $I_2$ respectively. As in \eqref{I_1 estimate 1} and \eqref{eqn 07.27.22.46}, one can prove
\begin{align*}
\tilde{I}_{1}\leq N t^{-\beta}\frac{(\phi(|x|^{-2}))^{-1}}{|x|^{d+m}} \leq N \int_{(\phi(|x|^{-2}))^{-1}}^{2t^{\alpha}} (\phi^{-1}(r^{-1}))^{(d+m)/2} t^{-\beta} dr.
\end{align*}
One can handle $\tilde{I}_{2}$ as in  \eqref{eqn 07.27.23:07} and  prove
$$
\tilde{I}_2 \leq N \int_{(\phi(|x|^{-2}))^{-1}}^{2t^{\alpha}} (\phi^{-1}(r^{-1}))^{d/2} t^{-\beta} dr.
$$
Finally we consider $II$. Note that \eqref{I_2 estimate 1} holds even if $\beta\notin \bN$. Hence, by \eqref{I_2 estimate 2},
\begin{align*}
t^{\alpha-\beta} (\phi^{-1}(t^{-\alpha}))^{(d+m)/2} &\leq N \int_{t^\alpha}^{2t^{\alpha}} (\phi^{-1}(r^{-1}))^{(d+m)/2} t^{-\beta} dr
\\
&\leq N \int_{(\phi(|x|^{-2}))^{-1}}^{2t^{\alpha}} (\phi^{-1}(r^{-1}))^{(d+m)/2} t^{-\beta} dr.
\end{align*}
Thus (ii) is also proved.
\end{proof}

\begin{corollary} \label{qintegrable}
Let $\alpha\in (0,1)$ and $\beta\in \bR$.

(i) There exists a constant $N=N(\alpha,\beta, d, c, \delta_0)$ such that
\begin{equation*}
\int_{\bR^d} |q_{\alpha,\beta}(t,x)|dx \leq N t^{\alpha-\beta}, \quad t>0.
\end{equation*}

(ii) For any $0<\varepsilon<T<\infty$,
\begin{equation*}
\int_{\bR^d} \sup_{[\varepsilon,T]}  |q_{\alpha,\beta}(t,x)| dx < \infty.
\end{equation*}

\end{corollary}

\begin{proof}

(i) Due to the similarity, we only consider the case $\beta\in\bN$. 
 By Lemma \ref{int phi}, Lemma \ref{qestimatebeta123}, and Fubini's theorem,
\begin{align*}
\int_{\bR^d} |q_{\alpha,\beta}(t,x)|dx =&\int_{|x|\geq (\phi^{-1}(t^{-\alpha}))^{-1/2}} |q_{\alpha,\beta}(t,x)|dx
\\
&+\int_{|x|< (\phi^{-1}(t^{-\alpha}))^{-1/2}} |q_{\alpha,\beta}(t,x)|dx
\\
\leq& N \int_{|x|\geq (\phi^{-1}(t^{-\alpha}))^{-1/2}} t^{2\alpha-\beta}\frac{\phi(|x|^{-2})}{|x|^d} dx
\\
&+ N \int_{|x|< (\phi^{-1}(t^{-\alpha}))^{-1/2}} \int_{(\phi(|x|^{-2}))^{-1}}^{2t^{\alpha}}(\phi^{-1}(r^{-1}))^{d/2} rt^{-\alpha-\beta} dr dx
\\
\leq& N \int_{r\geq \left( \phi^{-1}(t^{-\alpha}) \right)^{-1/2}} t^{2\alpha-\beta} \frac{\phi(r^{-2})}{r}dr
\\
&+ N \int_{0}^{2t^\alpha} \int_{(\phi(|x|^{-2}))^{-1}\leq r} (\phi^{-1}(r^{-1}))^{d/2} rt^{-\alpha-\beta} dx dr
\\
\leq& N t^{\alpha-\beta} + N \int_0^{2t^\alpha} rt^{-\alpha-\beta} dr
\leq N t^{\alpha-\beta}.
\end{align*}

(ii) Again we only prove the case  $\beta\in\bN$. Let $0<\varepsilon<T<\infty$. Since $t^{2\alpha-\beta}\leq N(\varepsilon,T,\alpha,\beta)$ for $t\in[\varepsilon,T]$, by Lemma \ref{qestimatebeta123},
\begin{align*}
|q_{\alpha,\beta}(t,x)|\leq N(\alpha,\beta,d,c,\delta_0,\varepsilon,T) \frac{\phi(|x|^{-2})}{|x|^d}, \quad t\in[\varepsilon,T].
\end{align*}

Also, if  $\varepsilon^\alpha \phi(|x|^{-2})\geq1$, and $t\in[\varepsilon,T]$, then   using Lemma \ref{qestimatebeta123} again, we get
\begin{align*}
|q_{\alpha,\beta}(t,x)|\leq N(\alpha,\beta,d,c,\delta_0,\varepsilon,T) \int_{(\phi(|x|^{-2}))^{-1}}^{2T^\alpha} (\phi^{-1}(r^{-1}))^{d/2} r dr.
\end{align*}

\noindent
As in the proof of (i),
\begin{align*}
\int_{\bR^d} \sup_{[\varepsilon,T]}|q_{\alpha,\beta}(t,x)|dx =&\int_{|x|\geq (\phi^{-1}(\varepsilon^{-\alpha}))^{-1/2}} \sup_{[\varepsilon,T]}|q_{\alpha,\beta}(t,x)|dx
\\
&+\int_{|x|< (\phi^{-1}(\varepsilon^{-\alpha}))^{-1/2}} \sup_{[\varepsilon,T]}|q_{\alpha,\beta}(t,x)|dx
\\
\leq& N \int_{|x|\geq (\phi^{-1}(\varepsilon^{-\alpha}))^{-1/2}} \frac{\phi(|x|^{-2})}{|x|^d} dx
\\
&+ N \int_{|x|< (\phi^{-1}(\varepsilon^{-\alpha}))^{-1/2}} \int_{(\phi(|x|^{-2}))^{-1}}^{2T^{\alpha}}(\phi^{-1}(r^{-1}))^{d/2} r dr dx
\\
\leq& N + N \int_{0}^{2T^\alpha} \int_{(\phi(|x|^{-2}))^{-1}\leq r} (\phi^{-1}(r^{-1}))^{d/2} r dx dr
\\
\leq& N + N \int_0^{2T^\alpha} r dr <\infty.
\end{align*}
The corollary is proved.\end{proof}

\mysection{Key estimates: BMO and \texorpdfstring{$L_q(L_p)$}{Lq(Lp)}-estimates}

In this section we prove some a priori estimates for solutions to the  equation with zero initial condition
\begin{equation}\label{mainequation-1}
\partial_t^\alpha u = \phi(\Delta)u + f,\quad t>0\,; \quad u(0,\cdot)=0.
\end{equation}

We first present the representation formula.

\begin{lemma} \label{u=qfsolution}

(i) Let $u\in C_c^\infty(\bR_+^{d+1})$ and denote $f:=\partial_t^\alpha u-\phi(\Delta)u$. Then
\begin{equation} \label{u=qf}
u(t,x)=\int_0^t \int_{\bR^d} q_{\alpha,1}(t-s,x-y)f(s,y) dy ds.
\end{equation}

(ii) Let $f\in C_c^\infty(\bR_+^{d+1})$ and define $u$ as in (\ref{u=qf}). Then $u$ satisfies equation \eqref{mainequation-1} for each $(t,x)$.

\end{lemma}

\begin{proof}
 See \cite[Lemma 3.5]{kim17timefractionalpde}, which treats the case $\phi(\lambda)=\lambda$. The proof for the general case is same. The only difference is one needs to use   formula \eqref{fourier of q} and Corollary \ref{qintegrable}  in place of their corresponding results when $\phi(\lambda)=\lambda$.
\end{proof}

Recall that due to Corollary \ref{qintegrable} (i), $q_{\alpha,1}$ is integrable on $[0,T]\times \bR^d$ for any $T>0$.  Also note that  
$\phi(\Delta)f$ is bounded for any $f\in C_c^\infty (\bR^{d+1})$. Therefore   the  operators $\cG_0$ and $\cG$ are well defined on $C_c^\infty (\bR^{d+1})$: 
\begin{equation*}
\cG_0 f:=\int_{-\infty}^t \int_{\bR^d} q_{\alpha,1}(t-s,y)f(s,x-y) dy ds,
\end{equation*}
\begin{equation*}
\cG f:=\phi(\Delta)\cG_0 f=\int_{-\infty}^t \int_{\bR^d} q_{\alpha,1}(t-s,y)\phi(\Delta)f(s,x-y) dy ds.
\end{equation*}

For each fixed $s$ and $t$ such that  $s<t$, define
$$
T_{t,s}f(x) :=\int_{\bR^d} q_{\alpha,1}(t-s,x-y) \phi(\Delta) f(s,y) dy,
$$
and
$$
G_{t,s}f(x) :=\int_{\bR^d} q_{\alpha,1+\alpha}(t-s,x-y)  f(s,y) dy.
$$
Note that, by Corollary \ref{qintegrable},  $T_{t,s}f$ and $G_{t,s}f$ are square integrable. Moreover, from \eqref{Mittag recurr} and \eqref{fourier of q} we have
\begin{equation*}
\begin{aligned}
\cF_d \{T_{t,s}f\}(\xi) & = -\phi(|\xi|^2) \hat{q}_{\alpha,1}(t-s,\xi) \hat{f}(s,\xi)
\\
&=\hat{q}_{\alpha,1+\alpha}(t-s,\xi)\hat{f}(s,\xi)=\cF_d \{G_{t,s}f\}(\xi).
\end{aligned}
\end{equation*}
This implies that $T_{t,s}f=G_{t,s}f$ for $s<t$.  Therefore, we have
\begin{equation*}
       \label{eqn 07.06.16:02}
\cG f(t,x)=\phi(\Delta)\cG_0 f=\lim_{\varepsilon \downarrow 0} \int_{-\infty}^{t-\varepsilon} \left(\int_{\bR^d} q_{\alpha,1+\alpha}(t-s,y) f(s,x-y) dy\right) ds.
\end{equation*}

The following result concerns the $L_2$-boundedness of the operator $\cG$.

\begin{lemma} \label{22estimate}
There exists a constant $N$ depending only on $\alpha$ and $d$ such that for any $f\in C_c^\infty(\bR^{d+1})$
\begin{equation} \label{L2estimate}
\|\cG f\|_{L_2(\bR^{d+1})}\leq N \|f\|_{L_2(\bR^{d+1})}.
\end{equation}
 Consequently,   operator $\cG$ is continuously extended to $L_2(\bR^{d+1})$. 
\end{lemma}

\begin{proof}
We follow the proof of   \cite[Lemma 3.6]{kim17timefractionalpde} which treats the case $\phi(\lambda)=\lambda$.

Denote $q_M:=q_{\alpha,1}\mathbf{1}_{0<t<M}$. By Corollary \ref{qintegrable},  $q_M$ is integrable in $\bR^{d+1}$. Denote 
$$\cG_M f = q_M * \phi(\Delta)f.
$$
 Then by Parseval's identity
\begin{align}
 \label{eqn 7.29.6}
\|\cG_M f\|_{L_2}^2=N(d)\int_{\bR^{d+1}}|\cF_{d+1}(q_M * \phi(\Delta)f)|^2 d\tau d\xi,
\end{align}
where $\cF_{d+1}$ is the $(d+1)$-dimensional Fourier transform. By the properties of the Fourier transform,
\begin{eqnarray}
\nonumber
\cF_{d+1}(q_M * \phi(\Delta)f)(\tau,\xi)&=&-N(d)\phi(|\xi|^2)\cF_{d+1}(q_M)(\tau,\xi)\cF_{d+1}(f)(\tau,\xi)\\
&=&-N(d)J_M(\tau,\xi)\cF_{d+1}(f)(\tau,\xi),   \label{eqn 7.30.1}
\end{eqnarray}
where 
\begin{align*}
J_M(\tau,\xi):=&-\phi(|\xi|^2)\cF_{d+1}(q_M)(\tau,\xi)\\
=& \int_0^M -\phi(|\xi|^2) e^{-i\tau t}\cF_d(q_M)(t,\xi)dt
\\
=& \int_0^M -\phi(|\xi|^2) e^{-i\tau t}D_t^{1-\alpha}E_\alpha(-t^\alpha\phi(|\xi|^2))dt.
\end{align*}
The last equality is due to  (\ref{fourier of q}). By \cite[Proposition 3.25]{gorenflo2014mittag},
\begin{equation*}
-\phi(|\xi|^2)D_t^{1-\alpha}E_\alpha(-t^\alpha\phi(|\xi|^2)) = \frac{d}{dt}E_\alpha (-t^\alpha\phi(|\xi|^2)).
\end{equation*}
Therefore,
\begin{align*}
\nonumber
J_M(\tau,\xi)=& \int_0^M e^{-i\tau t}\frac{d}{dt}E_\alpha (-t^\alpha\phi(|\xi|^2)) dt
\\
\nonumber
=& \int_0^M i\tau e^{-i\tau t} E_\alpha(-t^\alpha\phi(|\xi|^2))dt + e^{-i\tau M}E_\alpha(-M^\alpha\phi(|\xi|^2))-1
\\
\nonumber
=& \int_0^{|\tau| M} i\text{sgn}(\tau) e^{-\text{sgn}(\tau)it} E_\alpha(-(t/|\tau|)^\alpha\phi(|\xi|^2))dt 
\\
\label{I_M formula}
&\quad+ e^{-i\tau M}E_\alpha(-M^\alpha\phi(|\xi|^2))-1\\
=&I_M(\tau,\sqrt{\phi(|\xi|^2)}),
\end{align*}
where $I_M(\tau,\lambda)$ is taken from  formula (3.17) in \cite{kim17timefractionalpde}, and  by the inequality above (3.16) in \cite{kim17timefractionalpde}
$$
\sup_{M,\lambda>0,\tau\in \bR} |I_M(\tau,\lambda)|<\infty.
$$
This,  \eqref{eqn 7.29.6}, \eqref{eqn 7.30.1} and   Fatou's lemma lead to \eqref{L2estimate}. The lemma is proved. 
\end{proof}

We define an increasing function $\kappa:(0,\infty)\to (0,\infty)$ by 
$$
\kappa(b)=(\phi(b^{-2}))^{-1/\alpha}, \quad b>0.
$$
For  $(t,x)\in\R^{d+1}$ and $b>0$,  denote
\begin{equation*}
 Q_b(t,x)=(t-\kappa(b),\,t+\kappa(b))\times {B}_b(x),
\end{equation*}
and
\begin{equation*}
Q_b=Q_b(0,0), \quad B_b=B_b(0).
\end{equation*}
For  measurable subsets $Q\subset \R^{d+1}$ with  finite measure and  locally integrable functions $h$, define
\begin{equation*}
h_Q=\aint_{Q}h(s,y)dyds=\frac{1}{|Q|}\int_{Q}h(s,y)dyds,
\end{equation*}
where $|Q|$ is the Lebesgue measure of $Q$.

In the following lemmas, Lemmas \ref{ininestimate} - \ref{outinspaceestimate},  we estimate the mean oscillation of $\cG f$ on $Q_b$. For this, 
we  consider the  cases

\begin{itemize}

\item (Lemma  \ref{inwholeestimate}) $f$ has support in   $(-3\kappa(b), \infty)\times \bR^d$; 
 \item (Lemma \ref{outwholetimeestimate}) $f$ has support in $(-\infty, -2\kappa(b))\times \bR^d$.

\end{itemize}

The second case above is further divided into the  cases

\begin{itemize}

\item (Lemma \ref{outoutspaceestimate}) $f$ has  support in  $(-\infty,-2\kappa(b))\times B_{2b}^c$; 
\item (Lemma \ref{outinspaceestimate}) $f$ has support in  $(-\infty,-2\kappa(b))\times B_{3b}$.

\end{itemize}

\begin{lemma} \label{ininestimate}
Let $f\in C^{\infty}_c(\bR^{d+1})$  have a support in   $ (-3\kappa(b), 3\kappa(b))\times B_{3b}$ for some $b>0$. Then,
\begin{align*}
\aint_{Q_b}|\cG f (t,x)|dxdt \leq N \|f\|_{L_\infty(\bR^{d+1})},
\end{align*}
where $N$ depends only on $\alpha$ and $d$.
\end{lemma}

\begin{proof}
By the assumption and \eqref{phiratio}
$$
\|f\|_{L_2(\bR^{d+1})}\leq N(\alpha,d) |Q_{b}|^{1/2}\|f\|_{L_{\infty}(\bR^{d+1})}.
$$
Thus,  by H\"older's inequality and Lemma \ref{22estimate},
\begin{align*}
\int_{Q_b}|\cG f (t,x)|dxdt \leq & \left(\int_{Q_b}|\cG f (t,x)|^2dxdt\right)^{1/2}|Q_b|^{1/2}
\\
\leq & \|\cG f\|_{L_2(\bR^{d+1})}|Q_b|^{1/2}
\\
\leq & N\|f\|_{L_2(\bR^{d+1})}|Q_b|^{1/2}
\\
\leq & N \|f\|_{L_\infty(\bR^{d+1})}|Q_b|.
\end{align*}
The lemma is proved.
\end{proof}

Here is a generalization of Lemma \ref{ininestimate}.

\begin{lemma} \label{inwholeestimate} 
Let $f\in C_c^\infty(\bR^{d+1})$  have a support in $(-3\kappa(b), \infty)\times \bR^d$ for some $b>0$.  Then,
\begin{align*}
\aint_{Q_b}|\cG f (t,x)|dxdt \leq N \|f\|_{L_\infty(\bR^{d+1})},
\end{align*}
where $N$ depends only on $\alpha,d,c$ and $\delta_0$.
\end{lemma}

\begin{proof}
Take $\zeta_0=\zeta_0(t)\in C^{\infty}(\bR)$ such that $0\leq \zeta_0\leq 1$, $\zeta_0(t)=1$ for $t\leq 2\kappa(b)$, and $\zeta_0(t)=0$ for $t\geq 5/2\kappa(b)$.
Note that $\cG f=\cG (f\zeta_0)$ on $Q_b$ and $|f\zeta_0|\leq |f|$. This implies that to prove the lemma it is enough to assume $f(t,x)=0$ if $|t|\geq 3\kappa(b)$.

Choose a function $\zeta=\zeta(x) \in C_c^\infty (\bR^d)$ such that $\zeta=1$ in $B_{7b/3}$, $\zeta=0$ outside of $B_{8b/3}$ and $0\leq\zeta\leq1$. Set $f_1=\zeta f$ and $f_2=(1-\zeta)f$. Then $\cG f = \cG f_1 + \cG f_2$. Since $\cG f_1$ can be estimated by Lemma \ref{ininestimate}, to prove the lemma, we may further assume that $f(t,y)=0$ if $y\in B_{2b}$.  Therefore, for any $x\in B_b$,
\begin{align*}
\int_{\bR^d} \left|q_{\alpha,1+\alpha} (t-s,x-y) f(s,y)\right| dy =& \int_{|y-x|\geq 2b} \left|q_{\alpha,1+\alpha} (t-s,y) f(s,x-y)\right| dy 
\\
\leq & \int_{|y|\geq b} \left|q_{\alpha,1+\alpha} (t-s,y) f(s,x-y) \right| dy.
\end{align*}
 By \eqref{int phi} and   \eqref{eqn 7.21.7},
\begin{align*}
& \int_{|y|\geq b} \left|q_{\alpha,1+\alpha} (t-s,y) f(s,x-y) \right| dy 
\\
&\leq \|f\|_{L_\infty(\bR^{d+1})} 1_{|s|\leq 3\kappa(b)} \int_{|y|\geq b} |q_{\alpha,1+\alpha} (t-s,y)| dy
\\
&\leq N\|f\|_{L_\infty(\bR^{d+1})}1_{|s|\leq 3\kappa(b)} \int_b^\infty (t-s)^{\alpha-1}\frac{\phi(\rho^{-2})}{\rho^{d}}\rho^{d-1}d\rho
\\
&\leq N\|f\|_{L_\infty(\bR^{d+1})} 1_{|s|\leq 3\kappa(b)}  (t-s)^{\alpha-1}\phi(b^{-2}).
\end{align*}
Note that if $|t| \leq \kappa(b)$ and $|s|\leq 3\kappa(b)$ then $|t-s|\leq 4\kappa(b)$.    It follows that for any $(t,x)\in Q_{b}$,  
\begin{align*}
|\cG f(t,x)| &\leq N\|f\|_{L_\infty(\bR^{d+1})} \phi(b^{-2}) \int_{|t-s|\leq 4\kappa(b)} |t-s|^{\alpha-1} ds
\\
& \leq N \|f\|_{L_{\infty}(\R^{d+1})}.
\end{align*}
 This implies the desired estimate.
\end{proof}

\begin{lemma}
\label{outwholetimeestimate}
Let $f\in C_c^\infty(\bR^{d+1})$  have a support in  $(-\infty, -2\kappa(b))\times \bR^d$ for some $b>0$. Then, for  any $(t_1,x), (t_2,x)\in Q_b$, 
\begin{align*}
|\cG f (t_1,x)-\cG f(t_2,x)|   \leq N \|f\|_{L_\infty(\bR^{d+1})},
\end{align*}
where $N$ depends only on $\alpha,d,c$ and $\delta_0$. In particular, we have
$$
\aint_{Q_b}\aint_{Q_b} |\cG f (t_1,x)-\cG f(t_2,x)| dx dt_1 d\tilde{x} dt_2\leq N \|f\|_{L_\infty(\bR^{d+1})}.
$$
 \end{lemma}

\begin{proof}
Without loss of generality,   we assume $t_{1}>t_{2}$. Then, since $f(s,x)=0$ for $s\geq -2\kappa(b)$ and $t_1, t_2\geq -\kappa(b)$, it follows that
\begin{align*}
&|\cG f(t_1,x)-\cG f(t_2,x)|
\\
&=\Big|\int_{-\infty}^{t_1}  \int_{\bR^d} q_{\alpha,1+\alpha} (t_1-s,y)f(s,x-y) dy ds\\
&\quad \quad  - \int_{-\infty}^{t_2}  \int_{\bR^d} q_{\alpha,1+\alpha} (t_2-s,y)f(s,x-y) dr dy ds \Big|
\\
&= \Big|\int_{-\infty}^{-2\kappa(b)}  \int_{\bR^d} \left(q_{\alpha,1+\alpha} (t_1-s,x-y)-q_{\alpha,1+\alpha} (t_2-s,x-y)\right)f(s,y) dy ds \Big|.
\end{align*}
 By the fundamental theorem of calculus and \eqref{eqn 07.28 20:25}, we have
\begin{align*}
&\Big|\int_{-\infty}^{-2\kappa(b)}  \int_{\bR^d} \left(q_{\alpha,1+\alpha} (t_1-s,x-y)-q_{\alpha,1+\alpha} (t_2-s,x-y)\right)f(s,y)  dy ds \Big|
\\
&= \Big|\int_{-\infty}^{-2\kappa(b)}  \int_{\bR^d}  \int_{t_2}^{t_1} q_{\alpha,2+\alpha} (t-s,x-y) f(s,y)  dt dy ds \Big|.
\end{align*}
By Corollary \ref{qintegrable} (i),
\begin{align*}
\int_{\bR^d} \int_{t_2}^{t_1} |q_{\alpha,2+\alpha} (t-s,x-y)f(s,y)| dt dy 
\leq N \|f\|_{L_\infty(\bR^{d+1})} \int_{t_2}^{t_1} (t-s)^{-2}dt.
\end{align*}
Therefore,  if  $-\kappa(b)\leq t_2<t_1 \leq \kappa(b)$, 

\begin{align*}
&\Big|\int_{-\infty}^{-2\kappa(b)}  \int_{\bR^d}  \int_{t_2}^{t_1} q_{\alpha,2+\alpha} (t-s,x-y) f(s,y)  dt dy ds\Big|
\\
&\leq N \|f\|_{L_\infty(\bR^{d+1})} \left(\int_{t_2}^{t_1}\int_{-\infty}^{-2\kappa(b)} (t-s)^{-2}ds dt \right)
\\
&\leq N \|f\|_{L_\infty(\bR^{d+1})} \left(\int_{t_2}^{t_1} \frac{1}{\kappa(b)} dt \right)
\leq N \|f\|_{L_\infty(\bR^{d+1})}.
\end{align*}
This certainly proves the lemma. 
\end{proof}

\begin{lemma}
\label{outoutspaceestimate}
Let $f\in C_c^\infty(\bR^{d+1})$     have a support in  $(-\infty,-2\kappa(b))\times B_{2b}^c$ for some $b>0$. Then, for any $(t,x_1), (t,x_2)\in Q_b$, 
\begin{align*}
|\cG f (t,x_1)-\cG f(t,x_2)| \leq N \|f\|_{L_\infty(\bR^{d+1})},
\end{align*}
where $N$ depends only on $\alpha, d,  c$ and $\delta_0$. In particular, we have
$$
\aint_{Q_b}\aint_{Q_b} |\cG f (t,x_1)-\cG f(t,x_2)| dx_1 dt  dx_2 d \tilde{t} \leq N \|f\|_{L_\infty(\bR^{d+1})}.
$$
\end{lemma}

\begin{proof}
Recall $f(s,y)=0$ if $s\geq -2\kappa(b)$ or $|y|\leq 2b$. Thus,  if $t>-\kappa(b)$,  
\begin{align*}
&|\cG f(t,x_1)-\cG f(t,x_2)|
\\
&=\Big|\int_{-\infty}^{-2\kappa(b)}  \int_{|y|\geq 2b} \left(q_{\alpha,1+\alpha} (t-s,x_1-y)-q_{\alpha,1+\alpha} (t-s,x_2-y)\right) f(s,y) dy ds \Big|.
\end{align*}
By the fundamental theorem of calculus, for any $x_1, x_2 \in B_ b$ and $t>-\kappa(b)$, 
\begin{align}
&\Big|\int_{-\infty}^{-2\kappa(b)}  \int_{|y|\geq 2b} \left(q_{\alpha,1+\alpha} (t-s,x_1-y)-q_{\alpha,1+\alpha} (t-s,x_2-y)\right) f(s,y) dy ds \Big|   \nonumber
\\
&\leq \Big|\int_{-\infty}^{-2\kappa(b)}  \int_{|y|\geq 2b} \int_0^1 \nabla q_{\alpha,1+\alpha} (t-s, \theta(x_{1},x_{2},u)-y)\cdot (x_2-x_1) f(s,y) du dy ds \Big|   \nonumber
\\
& \leq N b \|f\|_{L_\infty(\bR^{d+1})}  \int_{-\infty}^{ -2\kappa(b)} \int_{|y|\geq b} |\nabla q_{\alpha,1+\alpha} (t-s,y)| dy ds  \nonumber
\\
& \leq N b \|f\|_{L_\infty(\bR^{d+1})}  \int_{\kappa(b)}^{\infty} \int_{|y|\geq b} |\nabla q_{\alpha,1+\alpha} (s,y)| dy ds,  \label{eqn 8.2.1}
\end{align}
 where $\theta(x_{1},x_{2},u)=(1-u)x_{1}+ux_{2}$.   By Lemma \ref{qestimatebeta123},
\begin{align*}
&\int_{\kappa(b)}^{\infty} \int_{|y|\geq b} |\nabla q_{\alpha,1+\alpha} (s,y)| dy ds
\\
&\leq N \int_{\kappa(b)}^\infty \int_{\left(\phi^{-1}(s^{-\alpha})\right)^{-1/2}}^\infty s^{\alpha-1} \frac{\phi(\rho^{-2})}{\rho^2} d\rho ds
\\
& \quad + N\int_{\kappa(b)}^\infty \int_b^{\left(\phi^{-1}(s^{-\alpha})\right)^{-1/2}} \int_{\left(\phi(\rho^{-2}) \right)^{-1}}^{2s^\alpha} \left(\phi^{-1}(r^{-1}) \right)^{(d+1)/2} s^{-\alpha-1} \rho^{d-1} dr d\rho ds.
\end{align*}
We now estimate the last two integrals above.  First, by \eqref{int phi},
\begin{align}
& \int_{\left(\phi^{-1}(s^{-\alpha})\right)^{-1/2}}^\infty s^{\alpha-1} \frac{\phi(\rho^{-2})}{\rho^2} d\rho    \nonumber
\\
&\leq \left(\phi^{-1}(s^{-\alpha})\right)^{1/2} \int_{\left(\phi^{-1}(s^{-\alpha})\right)^{-1/2}}^\infty s^{\alpha-1} \frac{\phi(\rho^{-2})}{\rho} d\rho   \nonumber
\\
&\leq \left(\phi^{-1}(s^{-\alpha})\right)^{1/2} s^{-1}.  \nonumber 
\end{align}
Therefore, by  the change of the variables $s^{\alpha}\to s$,
\begin{align}
\nonumber
&\int_{\kappa(b)}^\infty \int_{\left(\phi^{-1/2}(s^{-\alpha})\right)^{-1/2}}^\infty s^{\alpha-1} \frac{\phi(\rho^{-2})}{\rho^2} d\rho ds
\\
&\leq N \int_{\kappa(b)}^\infty \left(\phi^{-1}(s^{-\alpha})\right)^{1/2} s^{-1} ds
= N \int_{(\kappa(b))^{\alpha}}^\infty \left(\phi^{-1}(s^{-1})\right)^{1/2} s^{-1} ds.   \label{eqn 8.2.6}
\end{align}
Second, by Fubini's theorem,
\begin{align}
&\int_{\kappa(b)}^\infty \int_b^{\left(\phi^{-1}(s^{-\alpha})\right)^{-1/2}} \int_{\left(\phi(\rho^{-2}) \right)^{-1}}^{2s^\alpha} \left(\phi^{-1}(r^{-1}) \right)^{(d+1)/2} s^{-\alpha-1} \rho^{d-1} dr d\rho ds    \nonumber
\\
&\leq \int_{\kappa(b)}^\infty  \int_{(\kappa(b))^\alpha}^{2s^\alpha} \int_b^{\left(\phi^{-1}(r^{-1})\right)^{-1/2}} \left(\phi^{-1}(r^{-1}) \right)^{(d+1)/2} s^{-\alpha-1} \rho^{d-1} d\rho dr ds    \nonumber
\\
&\leq \int_{\kappa(b)}^\infty \int_{(\kappa(b))^\alpha}^{2s^\alpha} \left(\phi^{-1}(r^{-1}) \right)^{1/2} s^{-\alpha-1} dr ds   \nonumber
\\
&\leq \int_{(\kappa(b))^\alpha}^\infty \int_{(r/2)^{1/\alpha}}^\infty \left(\phi^{-1}(r^{-1}) \right)^{1/2} s^{-\alpha-1} ds dr     \nonumber
\\
&\leq N \int_{(\kappa(b))^\alpha}^\infty \left(\phi^{-1}(r^{-1})\right)^{1/2} r^{-1} dr.      \label{eqn 8.2.7}
\end{align}
Note that if  $s\geq (\kappa(b))^{\alpha}$, then by (\ref{phiratio}) with $R=\phi^{-1}((\kappa(b))^{-\alpha})$ and $r=\phi^{-1}(s^{-1})$, we have 
$$
\phi^{-1}(s^{-1})\leq s^{-1}(\kappa(b))^{\alpha}  \phi^{-1}((\kappa(b))^{-\alpha}).
$$
Theorefore,
\begin{align*}
\int_{(\kappa(b))^\alpha}^\infty \left(\phi^{-1}(s^{-1})\right)^{1/2} s^{-1} ds \leq & N \left( \phi^{-1}((\kappa(b))^{-\alpha})\right)^{1/2} (\kappa(b))^{\alpha/2} 
\int_{(\kappa(b))^\alpha}^\infty s^{-3/2} ds
\\
= & N \left( \phi^{-1}((\kappa(b))^{-\alpha})\right)^{1/2}=N b^{-1}.
\end{align*}
Combining this with  \eqref{eqn 8.2.6} and \eqref{eqn 8.2.7}, and going back to \eqref{eqn 8.2.1},  we get  
$$
|\cG f(t,x_1)-\cG f(t,x_2)|\leq N b \|f\|_{L_\infty(\bR^{d+1})} b^{-1}=N\|f\|_{L_\infty(\bR^{d+1})}. 
$$
Therefore, the lemma is proved.
\end{proof}

\begin{lemma} \label{outinspaceestimate}
Let $f\in C_c^\infty(\bR^{d+1})$  have a support in   $(-\infty,-2\kappa(b))\times B_{3b}$ for some $b>0$. Then for any $(t,x)\in Q_b$
$$
|\cG f(t,x)|\leq N  \|f\|_{L_\infty(\bR^{d+1})},
$$
where $N$ depends only on $\alpha,d, c$ and $\delta_0$. In particular, 
\begin{align*}
&\aint_{Q_b}|\cG f (t,x)|dx dt \leq N \|f\|_{L_\infty(\bR^{d+1})}.
\end{align*}

\end{lemma}

\begin{proof}
By  assumption,  $(t,x)\in Q_b$,
\begin{align*}
|\cG f(t,x)| \leq&\int_{-\infty}^{-2\kappa(b)} \int_{B_{3b}} |q_{\alpha,1+\alpha} (t-s,x-y)f(s,y)|dyds
\\
\leq& N \|f\|_{L_\infty} \int_{-\infty}^{-2\kappa(b)} \int_{B_{3b}} |q_{\alpha,1+\alpha} (t-s,x-y)|dyds
\\
\leq& N \|f\|_{L_\infty} \int_{\kappa(b)}^{\infty} \int_{B_{4b}} |q_{\alpha,1+\alpha} (s,y)|dyds \\
\leq& N \|f\|_{L_\infty} \left( I +II \right),\end{align*}
where
$$
I=\int^{\kappa(4b)}_{\kappa(b)} \int_{B_{4b}} |q_{\alpha,1+\alpha} (s,y)|dyds, \quad  \quad 
II=\int_{\kappa(4b)}^{\infty} \int_{B_{4b}} |q_{\alpha,1+\alpha} (s,y)|dyds.
$$
By Corollary \ref{qintegrable} (i) and \eqref{phiratio}  
\begin{align*}
I\leq& N \int^{\kappa(4b)}_{\kappa(b)} s^{-1}ds \leq N  
 \log\left(\frac{\kappa(4b)}{\kappa(b)} \right) 
\\
=& N \log\left(\frac{\phi(b^{-2})}{\phi(b^{-2}/16)} \right)
\\
\leq& N \log\left(16\right).
\end{align*}
 By Lemma \ref{qestimatebeta123},
$$
II\leq N \int_{\kappa(4b)}^{\infty} \int_{B_{4b}}\int_{(\phi(|y|^{-2}))^{-1}}^{2s^{\alpha}} (\phi^{-1}(r^{-1}))^{d/2} s^{-\alpha-1} dr dy ds.
$$
By Fubini's theorem, if $s>\kappa(4b)$,
\begin{align*}
&\int_{B_{4b}}\int_{(\phi(|y|^{-2}))^{-1}}^{2s^{\alpha}} (\phi^{-1}(r^{-1}))^{d/2} s^{-\alpha-1} dr dy
\\
&=\int_{B_{4b}}\int_{(\phi(|y|^{-2}))^{-1}}^{(\phi(b^{-2}/16))^{-1}} (\phi^{-1}(r^{-1}))^{d/2} s^{-\alpha-1} dr dy
\\
&\quad +\int_{B_{4b}}\int_{(\phi(b^{-2}/16))^{-1}}^{2s^{\alpha}} (\phi^{-1}(r^{-1}))^{d/2} s^{-\alpha-1} dr dy
\\
&\leq \int_0^{\left(\phi(b^{-2}/16) \right)^{-1}} \int_{|y|\leq \left(\phi^{-1}(r^{-1})\right)^{-1/2}} (\phi^{-1}(r^{-1}))^{d/2} s^{-\alpha-1} dy dr 
\\
& \quad + \int_{\left(\phi(b^{-2}/16) \right)^{-1}}^{2s^\alpha} \int_{B_{4b}} (\phi^{-1}(r^{-1}))^{d/2} s^{-\alpha-1} dy dr
\\
&\leq N \left(\phi(b^{-2}/16) \right)^{-1} s^{-\alpha-1} + N b^d   s^{-\alpha-1} \int_{\left(\phi(b^{-2}/16) \right)^{-1}}^{2s^\alpha} (\phi^{-1}(r^{-1}))^{d/2} dr.
\end{align*}
Obviously,
\begin{align*}
\int_{\kappa(4b)}^{\infty} \left(\phi(b^{-2}/16) \right)^{-1} s^{-\alpha-1}ds \leq \frac{1}{\alpha}.
\end{align*}
Also, by Fubini's theorem and \eqref{phiratio} with $R=b^{-2}/16$ and $r=\phi^{-1}(r^{-1})$
\begin{align*}
&\int_{\kappa(4b)}^{\infty} b^d   s^{-\alpha-1} \int_{\left(\phi(b^{-2}/16) \right)^{-1}}^{2s^\alpha} (\phi^{-1}(r^{-1}))^{d/2} dr ds
\\
&\leq \int_{\left(\phi(b^{-2}/16) \right)^{-1}}^{\infty} \int_{(r/2)^{1/\alpha}}^\infty b^d   s^{-\alpha-1} (\phi^{-1}(r^{-1}))^{d/2} ds dr
\\
&\leq N \int_{\left(\phi(b^{-2}/16) \right)^{-1}}^{\infty} b^d r^{-1} (\phi^{-1}(r^{-1}))^{d/2} dr
\\
& \leq N \int_{\left(\phi(b^{-2}/16) \right)^{-1}}^{\infty} \left(\phi(b^{-2}/16) \right)^{-d/2} r^{-d/2-1} dr \leq N.
\end{align*}
Therefore, $I, II$ are bounded by a constant independent of $b$,   and 
the lemma is proved.
\end{proof}

\begin{corollary}
\label{outwholeestimate}
Let $f\in C_c^\infty(\bR^{d+1})$ and $b>0$. Then,
\begin{align*}
\aint_{Q_b}\aint_{Q_b}|\cG f (t,x)-\cG f(s,y)|dtdxdsdy \leq N \|f\|_{L_\infty(\bR^{d+1})},
\end{align*}
where $N$ depends only on $\alpha,d, c$ and $\delta_0$.
\end{corollary}

\begin{proof}
{\textbf{Step 1}}. Suppose that $f$ has a support in $(-\infty,-2\kappa(b))\times \bR^d$.

 Take $\zeta \in C_c^\infty(\bR^{d})$ such that $0\leq \zeta \leq 1$, $\zeta=1$ on $B_{7b/3}$ and $\zeta=0$ outside of $B_{8b/3}$. Then,
\begin{align*}
|\cG f (t,x)-\cG f(s,y)| &\leq |\cG f (t,x)-\cG f(t,y)|+|\cG f (t,y)-\cG f(s,y)|
\\
&\leq |\cG (\zeta f) (t,x)-\cG (\zeta f)(t,y)|
\\
&\quad+|\cG (1-\zeta)f (t,x)-\cG (1-\zeta)f(t,y)|+|\cG f (t,y)-\cG f(s,y)| \\
&\leq |\cG (\zeta f) (t,x)|+|\cG (\zeta f)(t,y)|
\\
&\quad+|\cG (1-\zeta)f (t,x)-\cG (1-\zeta)f(t,y)|+|\cG f (t,y)-\cG f(s,y)|.
\end{align*}
Applying Lemma \ref{outwholetimeestimate}, Lemma \ref{outoutspaceestimate} and Lemma \ref{outinspaceestimate} to estimate $|\cG f (t,y)-\cG f(s,y)|$, $|\cG (1-\zeta)f (t,x)-\cG (1-\zeta)f(t,y)|$ and  $|\cG (\zeta f) (t,x)|+|\cG (\zeta f)(t,y)|$ respectively,  we get the desired estimate.

{\textbf{Step 2}}. General case.  Choose a function $\xi=\xi(t) \in C^\infty(\bR)$ such that $0\leq \xi\leq 1$, $\xi=1$ on $(-\infty, -8\kappa(b)/3)$ and $\xi(t)=0$ for $t\geq -7\kappa(b)/3$.  Let $f_1=\xi f$ and $f_2=(1-\xi) f$. Then $\cG f= \cG f_1 + \cG f_2$ in $\bR^{d+1}$.  Note that $f_1$ satisfies  the assumption in Step 1, and $f_2$ satisfies the condition in Lemma \ref{inwholeestimate}. Therefore,  
\begin{align*}
&\aint_{Q_b} |\cG f(t,x)-(\cG f)_{Q_b}| dtdx
\\
&\leq \aint_{Q_b} |\cG f_1(t,x)-(\cG f_1)_{Q_b}| dtdx + \aint_{Q_b} |\cG f_2(t,x)-(\cG f_2)_{Q_b}| dtdx
\\
&\leq N \|f\|_{L_\infty(\bR^{d+1})}.
\end{align*}
The corollary is proved.
\end{proof}

For  locally integrable functions $h$ on $\bR^{d+1}$, we define the BMO semi-norm of $h$ on $\bR^{d+1}$ as
\begin{equation*}
\|h\|_{BMO(\bR^{d+1})}=\sup_{Q\in \bQ} \aint_Q |h(t,x)-h_Q| dtdx
\end{equation*}
where $h_Q=\aaint_Q h(t,x)dtdx$ and
\begin{equation*}
\bQ:=\{Q_b(t_0,x_0) : b>0, (t_0,x_0)\in \bR^{d+1} \}.
\end{equation*}

For  measurable functions $h(t,x)$ on $\bR^{d+1}$, we define the sharp function
\begin{align*}
h^{\#}(t,x)=\sup_{(t,x)\in Q_b(r,z)} \aint_{Q_b(r,z)}|h(s,y)-h_{Q_b(r,z)}|dsdy.
\end{align*}

\begin{theorem}[Fefferman-Stein Theorem]
\label{feffermanstein}
For any $1<p<\infty$ and $h\in L_p(\bR^{d+1})$,
\begin{equation*}
N^{-1}\|h^{\#}\|_{L_p(\bR^{d+1})}\leq \|h\|_{L_p(\bR^{d+1})}\leq N\|h^{\#}\|_{L_p(\bR^{d+1})},
\end{equation*}
where $N>1$ depends on $\alpha,d,p,c$ and  $\delta_0$.
\end{theorem}

\begin{proof}
See  \cite[Theorem I.3.1, Theorem IV.2.2]{stein1993harmonic}.  We only remark that  due to (\ref{phiratio}),  the balls $Q_b(s,y)$ satisfy the conditions (i)-(iv) in  \cite[Section 1.1] {stein1993harmonic}:
\begin{enumerate}[(i)]
\item $Q_c(t,x)\cap Q_c(s,y)\neq\varnothing$ implies $Q_c(s,y)\subset Q_{N_1 c}(t,x)$;
\\
\item $|Q_{N_1 c}(t,x)|\leq N_2|Q_c(t,x)|$;
\\
\item $\cap_{c>0} \bar{Q}_c(t,x)=\{(t,x)\}$ and $\cup_c Q_c(t,x)=\bR^{d+1}$;
\\
\item for each open set $U$ and $c>0$, the function $(t,x)\to|Q_c(t,x)\cap U|$ is continuous.
\end{enumerate}
\end{proof}

Here is the main result of this section.

\begin{theorem}

(i) For any $f\in L_2(\bR^{d+1})\cap L_\infty(\bR^{d+1})$,
\begin{align} \label{bmoestimate}
\|\cG f\|_{BMO(\bR^{d+1})} \leq N(\alpha,d,c, \delta_0) \|f\|_{L_\infty(\bR^{d+1})}.
\end{align}

(ii) For any $p,q\in(1,\infty)$ and $f\in C_c^\infty (\bR^{d+1})$,
\begin{align} \label{qpestimate}
\|\cG f\|_{L_q(\bR;L_p(\R^{d}))} \leq N(\alpha,d,c, \delta_0,p) \|f\|_{L_q(\bR;L_p(\R^{d}))}.
\end{align}

\end{theorem}

\begin{proof}

We follow the proof of \cite[Theorem 4.4]{kim17timefractionalpde} which treats the case $\phi(\lambda)=\lambda$.

(i) Note that for any $(t_0,x_0)\in\bR^{d+1}$,
\begin{align*}
\cG f (t+t_0,x+x_0)&= \int_{-\infty}^{t+t_0}  \int_{\bR^d} q_{\alpha,1+\alpha}(t+t_0-s,x+x_0-y)f(s,y) dy ds
\\
&= \int_{-\infty}^{t}  \int_{\bR^d} q_{\alpha,1+\alpha}(t-s,x-y)f(s+t_0,x_0+y) dr dy ds
\\
&=\cG \left(f(\cdot+t_0,\cdot+x_0)\right)(t,x).
\end{align*}
Therefore, with $\tilde{f}(t,x):=f(t+t_0, x+x_0)$,
\begin{eqnarray*}
&&\aint_{Q_b(t_0,x_0)} |\cG f(t,x)-(\cG f)_{Q_b(t_0,x_0)}| dtdx\\
&&= \aint_{Q_b} |\cG \tilde{f}(t,x)-(\cG \tilde{f})_{Q_b}| dtdx.
\end{eqnarray*}
Due to this and the translation invariant property of   $L_{\infty}$-norm,   for the proof of   (i)  it suffices to prove  that  there exists $N=N(\alpha,d,c,\delta_{0})$ such that 
\begin{equation} \label{meanaverageineq}
\aint_{Q_b} |\cG f(t,x)-(\cG f)_{Q_b}| dtdx \leq N \|f\|_{L_\infty(\bR^{d+1})}, \quad b>0.
\end{equation}
Recall that we already have \eqref{meanaverageineq} due to Corollary \ref{outwholeestimate} if $f\in C_c^\infty (\bR^{d+1})$. 

Now we consider the general case, that is $f\in L_2(\bR^{d+1})\cap L_\infty(\bR^{d+1})$. We choose a sequence of functions $f_n\in C_c^\infty(\bR^{d+1})$ such that $\cG f_n \to \cG f \ (a.e.)$, and $\|f_n\|_{L_\infty(\bR^{d+1})}\leq \|f\|_{L_\infty(\bR^{d+1})}$. Then by Fatou's lemma,
\begin{align*}
&\aint_{Q_b} |\cG f(t,x)-(\cG f)_{Q_b}| dtdx
\\
& \leq \aint_{Q_b} \aint_{Q_b} |\cG f(t,x)-\cG f(s,y)| dtdxdsdy
\\
& \leq  \liminf_{n\to \infty} \aint_{Q_b} \aint_{Q_b} |\cG f_n(t,x)-\cG f_n(s,y)| dtdxdsdy
\\
& \leq N  \liminf_{n\to \infty} \|f_n\|_{L_\infty(\bR^{d+1})}
\leq N \|f\|_{L_\infty(\bR^{d+1})}.
\end{align*}

\textbf{Step 1.} We prove (\ref{qpestimate}) for the case $p=q$. First assume that $p\geq 2$. Then by Lemma \ref{22estimate} and Theorem \ref{feffermanstein}, for any $f\in L_2(\bR^{d+1})\cap L_\infty(\bR^{d+1})$, it holds that
\begin{equation*}
\|(\cG f)^{\#}\|_{L_2(\bR^{d+1})}\leq N \|f\|_{L_2(\bR^{d+1})}.
\end{equation*}
By (\ref{bmoestimate}),
\begin{equation*}
\|(\cG f)^{\#}\|_{L_\infty(\bR^{d+1})}\leq N \|f\|_{L_\infty(\bR^{d+1})}.
\end{equation*}
Note that the map $f \to (\cG f)^{\#}$ is sublinear since $\cG$ is linear. Hence by a version of the Marcinkiewicz interpolation theorem, for any $p\in [2,\infty)$ there exists a constant $N$ such that
\begin{equation*}
\|(\cG f)^{\#}\|_{L_p(\bR^{d+1})}\leq N \|f\|_{L_p(\bR^{d+1})}
\end{equation*}
for all $f\in L_2(\bR^{d+1})\cap L_\infty(\bR^{d+1})$. Finally by Theorem \ref{feffermanstein}, we get
\begin{equation*}
\|\cG f\|_{L_p(\bR^{d+1})}\leq N \|f\|_{L_p(\bR^{d+1})}.
\end{equation*}
Therefore (\ref{qpestimate}) is proved for $p\in[2,\infty)$.

Now let $p\in(1,2)$. Take $f,g\in C_c^\infty (\bR^{d+1})$ and $p'=\frac{p}{p-1}\in(2,\infty)$. By Parseval's identity, Fubini's theorem, and the change of variables,
\begin{align} \label{duality}
&\int_{\bR^{d+1}} g(t,x)\cG f(t,x) dx dt = \int_{\bR^{d+1}} g(t,x)\phi(\Delta) \cG_0f(t,x) dx dt \nonumber
\\
&=N(d) \int_{\bR^{d+1}} \cF_{d}(g)(t,\xi) \phi(|\xi|^2) \cF_{d}(\cG_0 f)(t,\xi) d\xi dt \nonumber
\\
&= \int_{\bR^{d+1}} \phi(\Delta)g(t,x) \int_{\bR^{d+1}} 1_{t-s>0}q_{\alpha,1}(t-s,x-y) f(s,y) dy ds dx dt \nonumber
\\
&= \int_{\bR^{d+1}} \int_{\bR^{d+1}} \phi(\Delta)g(-t,-x)  1_{s-t>0}q_{\alpha,1}(s-t,y-x) f(-s,-y) dx dt dy ds \nonumber
\\
&= \int_{\bR^{d+1}} \cG \tilde{g}(s,y) f(-s,-y)dyds,
\end{align}
where  $\tilde{g}(t,x)=g(-t,-x)$. By H\"older's inequality,
\begin{align*}
&\left|\int_{\bR^{d+1}} g(t,x)\cG f(t,x) dx dt\right| 
\\
&\leq \|f\|_{L_p(\bR^{d+1})}\|\cG \tilde{g}\|_{L_{p'}(\bR^{d+1})} \leq \|f\|_{L_p(\bR^{d+1})}\|g\|_{L_{p'}(\bR^{d+1})}.
\end{align*}
Since $g\in C_c^\infty (\bR^{d+1})$ is arbitrary,  we have $\cG f\in L_{p}(\R^{d+1})$ and (\ref{qpestimate}) is also proved for $p\in(1,2)$.

\textbf{Step 2.} Now we prove (\ref{qpestimate}) for general $p,q\in(1,\infty)$.  Define $q_{\alpha,\beta}(t):=0$ for $t\leq 0$. For each $(t,s)\in\bR^2$, we define the operator $\cK(t,s)$ as follows:
\begin{equation*}
\cK(t,s)f(x):=\int_{\bR^d} q_{\alpha,1+\alpha}(t-s,x-y)f(y) dy, \quad f\in C_c^\infty(\R^{d}).
\end{equation*}
Let $p\in(1,\infty)$. Then,
\begin{align*}
\|\cK(t,s)f\|_{L_p(\R^{d})}&=\left\|\int_{\bR^d} q_{\alpha,1+\alpha}(t-s,x-y)f(y) dy\right\|_{L_p(\R^{d})}
\\
&\leq \|f\|_{L_p(\R^{d})} \int_{\bR^d} |q_{\alpha,1+\alpha}(t-s,y)|dy \leq N(t-s)^{-1}\|f\|_{L_p(\R^{d})}.
\end{align*}
Hence the operator $\cK(t,s)$ is uniquenly extendible to $L_p(\R^{d})$ for $t\neq s$. Denote
\begin{equation*}
Q:=[t_0,t_0+\delta), \quad Q^*:=[t_0-\delta,t_0+2\delta), \quad  \delta>0.
\end{equation*}
Note that for $t\notin Q^*$ and $s_1,s_2\in Q$, we have
$$
|s_1-s_2|\leq\delta, \quad |t-(t_0+\delta)|\geq\delta.
$$
Thus, for any $f\in L_p(\R^{d})$ such that $\|f\|_{L_p(\R^{d})}=1$,
\begin{align*}
&\|\cK(t,s_1)f-\cK(t,s_2)f\|_{L_p(\R^{d})}
\\
&=\left\| \int_{\bR^d}  \left(q_{\alpha,1+\alpha}(t-s_1,x-y)-q_{\alpha,1+\alpha}(t-s_2,x-y)\right)f(y) dy\right\|_{L_p(\R^{d})}
\\
&\leq \|f\|_{L_p(\R^{d})} \int_{\bR^d} \left|q_{\alpha,1+\alpha}(t-s_1,x-y)-q_{\alpha,1+\alpha}(t-s_2,x-y)\right| dy dx
\\
&\leq N \int_{\bR^d} \int_0^1 |q_{\alpha,2+\alpha}(t-us_1+(1-u)s_2,x)| |s_1-s_2| du dx
\\
&\leq \frac{N|s_1-s_2|}{(t-(t_0+\delta))^2}.
\end{align*}
due to Corollary \ref{qintegrable} (i). Here, recall that $\cK(t,s)=0$ if $t\leq s$. Hence,
\begin{align*}
\|\cK(t,s_1)-\cK(t,s_2)\|_{\Lambda} \leq \frac{N|s_1-s_2|}{(t-(t_0+\delta))^2}.
\end{align*}
where $\|\cdot\|_{\Lambda}$ denotes the operator norm of $\Lambda$ on $L_p(\R^{d})$. Therefore,
\begin{align*}
&\int_{\bR\setminus Q^*} \|\cK(t,s_1)-\cK(t,s_2)\|_{\Lambda} dt \leq N \int_{\bR\setminus Q^*}\frac{|s_1-s_2|}{(t-(t_0+\delta))^2} dt
\\
&\leq N|s_1-s_2|\int_{|t-(t_0+\delta)|\geq \delta}\frac{1}{(t-(t_0+\delta))^2} dt \leq N\delta \int_\delta^\infty t^{-2}dt \leq N.
\end{align*}
Furthermore, by following the proof of \cite[Theorem 1.1]{krylov2001caideron}, one can easily check that for almost every $t$ outside of the support of $f\in C_c^\infty(\bR;L_p(\R^{d}))$,
\begin{equation*}
\cG f(t,x)=\int_{-\infty}^\infty \cK(t,s)f(s,x)ds
\end{equation*}
where $\cG$ denotes the  extension  to  $L_p(\bR^{d+1})$ which is verified in Step 1. Hence by the Banach space-valued version of the Calder\'on-Zygmund theorem (e.g. \cite[Theorem 4.1]{krylov2001caideron}), our assertion is proved for $1<q\leq p$.

For $1<p<q<\infty$, define $p'=\frac{p}{p-1}$ and $q'=\frac{q}{q-1}$. By (\ref{duality}) and H\"older's inequality,
\begin{align*} 
\left|\int_{\bR^{d+1}} g(t,x)\cG f(t,x) dx dt\right| &=\left| \int_\bR \left(\int_{\bR^{d}} \cG \tilde{g}(s,y) f(-s,-y)dy\right)ds\right|
\\
&\leq  \int_{\bR} \|f(-s,\cdot)\|_{L_p(\R^{d})}\|\cG \tilde{g}(s,\cdot)\|_{L_{p'}(\R^{d})} ds 
\\
&\leq N \|f\|_{L_q(\bR;L_p(\R^{d}))} \|g\|_{L_{q'}(\bR;L_{p'}(\R^{d}))}
\end{align*}
for any $g\in C_c^\infty(\bR^{d+1})$, where the last inequality holds due to $1<q'<p'$. Since $g$ is arbitrary, we have
\begin{equation*}
\|\cG f\|_{L_q(\bR;L_p(\R^{d}))}\leq N \|f\|_{L_q(\bR;L_p(\R^{d}))}.
\end{equation*}
The theorem is proved.
\end{proof}

\mysection{Homogeneous equation}

In this section we study the  homogeneous equation with non-zero initial condition 
\begin{equation}\label{eqn 06.15.16:05}
\partial^{\alpha}_{t}u=\phi(\Delta)u,\quad u(0,x)=u_{0}(x).
\end{equation}

We first show that $q$ is a fundamental solution to equation \eqref{eqn 06.15.16:05}.

\begin{lemma} \label{zero converge}
Let $u_0\in C_{p}^{\infty}(\bR^d)$, and define $u$ as
$$
u(t,x):=\int_{\bR^d} q(t,x-y)u_0(y)dy.
$$

(i) As as $t \downarrow 0$, $u(t,\cdot)$ converges to $u_0(\cdot)$ uniformly on $\bR^d$ and also  in $H^n_p$ for any $n\in \bN_0$.

(ii) $u\in C^{\alpha,\infty}_{p}([0,T]\times\bR^d)$ and $u$ satisfies $\partial^{\alpha}_{t}u=\phi(\Delta)u$ for $t>0$. 

\end{lemma}
\begin{proof}

(i) By \eqref{Laplace trans},  for any $t>0$,
$$
\int_{\bR^d} q(t,y)dy =\hat{q}(t,0) =E_{\alpha}(0)=1.
$$
For any $\delta >0$,
\begin{align*}
&\left|\int_{\bR^d} q(t,x-y)u_0(y)dy -u_0(x) \right|\\
&=\left|\int_{\bR^d} q(t,y)(u_0(x-y)-u_0(x))dy \right| \\
&\leq \int_{|y| \leq \delta} \left|q(t,y)(u_0(x-y)-u_0(x))\right|dy +\int_{|y| > \delta} \left|q(t,y)(u_0(x-y)-u_0(x))\right|dy  \\
&=: I_{1}(\delta,x) + I_{2}(\delta,x).
\end{align*}
Since $u_0 \in C_{p}^{\infty}(\bR^{d})$,  for any $\varepsilon >0$,   one can take a small $\delta$ so that $I_{1}(\delta,x) < \varepsilon$.
Moreover due to \eqref{int phi} and Lemma \ref{qestimatebeta123},  we have 
\begin{equation}
\label{eqn 8.4.1}
\int_{|y|\geq \delta} q(t,y)dy \leq N t^{\alpha}\phi(\delta^{-2}),
\end{equation}
and therefore
\begin{align*}
I_{2}(\delta,x)\leq N  t^{\alpha} \phi(\delta^{-2})\|u_0\|_{L_\infty} \to 0 \quad \text{as} \quad t \downarrow 0. 
\end{align*}
This shows $u(t,x)\to u_0(x)$ as $t\to 0$ uniformly in $x$. Also,
 \begin{eqnarray*}
 \|u(t)-u_0\|^p_{L_p} &\leq& 2^p \|I_1(\delta)\|^p_{L_p}+2^p  \|I_2(\delta)\|^p_{L_p} \\
 &\leq&N \sup_{|z|\leq \delta}\|u_{0}(\cdot-z)-u_{0}(\cdot)\|^p_{L_{p}}+ N\|u_{0}\|^p_{L_{p}}\int_{|y| > \delta}|q(t,y)| dy.
 \end{eqnarray*}
 This and \eqref{eqn 8.4.1} certainly proves the $L_p$ convergence. Considering  $D^n_x$, we have $u(t,\cdot)\to u_0$ in $H^n_p$ for any $n$.

(ii) By \eqref{fourier of q} (recall \eqref{eqn 8.30}),
$$
\partial_t^\alpha u(t,x)=\int_{\bR^d} q_{\alpha,2\alpha}(t,y)u_0(x-y)dy - \frac{1}{\Gamma(1-\alpha)}t^{-\alpha}u_0(x).
$$
By \eqref{Mittag recurr} and \eqref{fourier of q}, we have
\begin{align*}
\cF\left({\partial^{\alpha}_t u(t,\cdot)}\right)(\xi) &= \hat{q}_{\alpha,2\alpha}(t,\xi)\hat{u}_{0}(\xi)-\frac{1}{\Gamma(1-\alpha)}t^{-\alpha}\hat{u}_{0}(\xi)
\\
&=\phi(|\xi|^{2})\left( R^{-1} E_{\alpha,1-\alpha}(-R)-R^{-1}\frac{1}{\Gamma(1-\alpha)} \right)\hat{u}_{0}(\xi)
\\
&=-\phi(|\xi|^{2})\left( E_{\alpha,1}(-R) \right) \hat{u}_{0}(\xi)
\\
&=-\phi(|\xi|^2)\hat{q}_{\alpha,\alpha}(t,\xi)\hat{u}_{0}(\xi) 
\\
&=\cF\left({\phi(\Delta)u}(t,\cdot)\right)(\xi),
\end{align*}
where $R=R(t,\xi)=t^{\alpha}\phi(|\xi|^{2})$. This together with (i) implies $u$ satisfies \eqref{eqn 06.15.16:05}. 

Now we show that $u\in C^{\alpha,\infty}_{p}([0,T]\times\R^{d})$. Since
$$
D^m_x u(t,x):=\int_{\bR^d} q(t,x-y)D^m_x u_{0}(y)dy,
$$
we have $D^{m}_{x}u\in C([0,T];L_{p})$. Also since $u$ satisfies \eqref{eqn 06.15.16:05} one can easily have $\partial^{\alpha}_{t}u\in C([0,T];L_{p})$. The lemma is proved.
\end{proof}

Recall that   $\Psi\in\cS(\R^{d})$ is chosen  such that  $\hat{\Psi}(\xi)$ is supported in a strip $\{1/2\leq |\xi| \leq 2\}$, $\hat{\Psi} \geq 0$, and satisfies \eqref{eqn 06.23.10:19}. Also recall that we set a sequence of functions $\Psi_{j}$ satisfying \eqref{eqn 06.23.10:17}.

Define
\begin{equation}\label{def 06.15.09:36}
\begin{aligned}
q_{j}(t,x)&=(\Psi_{j}(\cdot)\ast q(t,\cdot))(x)
=\cF^{-1}(\hat{\Psi}(2^{-j}\cdot)\hat{q}(t,\cdot))(x)
\\
&=2^{jd}\cF^{-1}(\hat{\Psi}(\cdot)\hat{q}(t,2^{j}\cdot))(2^{j}x)
:=2^{jd}\bar{q}_{j}(t,2^{j}x).
\end{aligned}
\end{equation}

\begin{lemma}\label{lem 06.15.09:35}
Let $p>1$, and $0<\alpha<1$. Then there exists a constant $N$ depending only on $\alpha,d,$ and $\phi$ such that
\begin{equation}\label{bound_of_p_j}
\|q_{j}(t,\cdot)\|_{L_{1}}\leq N ( \phi(2^{2j})^{-\frac{1}{\alpha}}  t^{-1}\wedge 1), \quad t>0.
\end{equation}
\end{lemma}
\begin{proof}

By Corollary \ref{qintegrable} (i) we have $\|q(t,\cdot)\|_{L_{1}}\leq N$.  Due to this and the relation $\|q_j(t,\cdot)\|_{L_1}= \|\bar{q}_j(t,\cdot)\|_{L_1}$,  it only  remains to show
$$
\|\bar{q}_{j}(t,\cdot)\|_{L_{1}}\leq N   \phi(2^{2j})^{-\frac{1}{\alpha}} t^{-1}.
$$
By definition (see \eqref{def 06.15.09:36})
\begin{equation}
     \label{eqn 5.16.7}
\cF(\bar{q}_j)(t,\xi)=\hat{\Psi}(\xi)\hat{q}(t,2^{j}\xi).
\end{equation}
Using the relation $\hat{q}(t,\xi)=E_{\alpha}(-\phi(|\xi|^{2})t^{\alpha})$ (see \eqref{fourier of q}) and  the following integral representation of $E_{\alpha}$(e.g. \cite[Exercise 3.9.5]{gorenflo2014mittag})
\begin{equation*}\label{representation_of_E_a_1}
E_{\alpha}(-v)=\frac{\sin{\alpha \pi}}{\pi}\int_{0}^{\infty}\frac{r^{\alpha-1}}{r^{2\alpha}+2r^{\alpha} \cos{(\alpha\pi)}+1}\exp{(-rv^{1/\alpha})}rdr, \quad \forall v>0
\end{equation*}
 we have
\begin{eqnarray}
\nonumber
|\cF{\bar{p}_{j}}(t,\xi)|&\leq& N 1_{1/2 \leq |\xi|\leq 2}\int_{0}^{1}r^{\alpha-1}\exp{(-\phi(2^{2j}|\xi|^{2})^{\frac{1}{\alpha}}tr)}r\,dr
\\
&&+N 1_{1/2 \leq |\xi|\leq 2} \int_{1}^{\infty} r^{-\alpha-1} \exp{(-\phi(2^{2j}|\xi|^{2})^{\frac{1}{\alpha}}tr)}r\,dr.   \label{eqn 5.16.8}
\end{eqnarray}
Note that for any polynomial $Q(z)$ of degree $m$ and constant $\sigma>0$, we have
$$
Q(z) e^{-\sigma z}\leq N(\sigma,m)z^{-1}.
$$

\noindent
This together with \eqref{eqn 5.16.8}, and \eqref{phiratio} easily yields
\begin{equation*}
|\cF\bar{q}_{j}(t,\xi)|\leq N \phi(2^{2j})^{-\frac{1}{\alpha}}t^{-1}1_{1/2 \leq |\xi|\leq 2}.
\end{equation*}
Similarly, using \eqref{eqn 5.16.7}  and following above computations,  for any multi-index $\gamma$ we get
\begin{equation*}
|D_{\xi}^{\gamma}\cF\bar{q}_{j}(t,\xi)|\leq N(\alpha,\gamma,d,\phi)\phi(2^{2j})^{-\frac{1}{\alpha}}t^{-1}1_{1/2 \leq |\xi|\leq 2}.
\end{equation*}
Therefore, we finally have
\begin{equation*}
\begin{aligned}
\|\bar{q}_{j}(t,\cdot)\|_{L_{1}}&=\int_{\R^{d}}(1+|x|^{2d})^{-1}(1+|x|^{2d})|\bar{q}_{j}(t,x)|dx
\\
&\leq N \int_{\R^{d}}(1+|x|^{2d})^{-1}\sup_{\xi}|(1+\Delta_{\xi}^{d})\cF(\bar{q}_{j})(t,\xi)|dx
\\
&\leq N  \phi(2^{2j})^{-\frac{1}{\alpha}}t^{-1}.
\end{aligned}
\end{equation*}
The lemma is proved.
\end{proof}

\begin{theorem}\label{Besov_bound_2}
Let $p>1$,  $0<\alpha<1$ and $f\in\Ccinf(\R^{d})$. Then we have
\begin{equation}\label{Besov_bound_of_q}
\int_{0}^{T}\|q\ast f(t,\cdot)\|^{q}_{L_{p}} dt \leq N \|f\|^{q}_{B^{\phi,-\frac{2}{\alpha q}}_{p,q}},
\end{equation}
where the constant $N$ depends only on $\alpha,d,\phi,p,q,T$.
\end{theorem}
\begin{proof}
Note that 
\begin{equation*}\label{middle_of_Besov_bound_6}
\begin{aligned}
\hat{\Psi}_{j}&=\hat{\Psi}_{j}(\hat{\Psi}_{j-1}+\hat{\Psi}_{j}+\hat{\Psi}_{j+1}), \quad j\in\bN,
\\
\hat{\Psi}_{0}&=\hat{\Psi}_{0}(\hat{\Psi}_{0}+\hat{\Psi}_{1}).
\end{aligned}
\end{equation*} 
Using this and the relation $\cF(f_{1}\ast f_{2})=\cF(f_{1})\cF(f_{2})$,
\begin{equation*}
\begin{aligned}
\int_{0}^{T}\|q\ast f(t,\cdot)\|^{q}_{L_{p}} dt
&\leq N \int_{0}^{T}(\|q_{0}(t,\cdot)\|_{L_{1}}+\|q_{1}(t,\cdot)\|_{L_{1}})^{q}\|f_{0}\|^{q}_{L_{p}}dt
\\
&\quad+ N \int_{0}^{T}\big(\sum_{j=1}^{\infty}\sum_{i=j-1}^{j+1}\|q_{i}(t,\cdot)\|_{L_{1}}\|f_{j}\|_{L_{p}}\big)^{q}dt.
\end{aligned}
\end{equation*}

\noindent
By \eqref{bound_of_p_j}
\begin{equation}\label{eqn 07.09.1}
\int_{0}^{T}(\|q_{0}(t,\cdot)\|_{L_{1}}+\|q_{1}(t,\cdot)\|_{L_{1}})^{q}\|f_{0}\|^{q}_{L_{p}}dt \leq N(T) \|f_{0}\|_{L_{p}}^{q},
\end{equation}
and
\begin{equation*}
\int_{0}^{T}\big(\sum_{j=1}^{\infty}\sum_{i=j-1}^{j+1}\|q_{i}(t,\cdot)\|_{L_{1}}\|f_{j}\|_{L_{p}}\big)^{q}dt
\leq N \int_{0}^{T}\big(\sum_{j=1}^{\infty}(  \phi(2^{2j})^{-\frac{1}{\alpha}}  t^{-1}\wedge 1)\|f_{j}\|_{L_{p}}\big)^{q}dt.
\end{equation*}
Observe that
\begin{equation*}
\begin{aligned}
\int_{0}^{T}&\big(\sum_{j=1}^{\infty}( \phi(2^{2j})^{-\frac{1}{\alpha}}  t^{-1}\wedge 1)\|f_{j}\|_{L_{p}}\big)^{q}dt
\\
&\leq 2^q\int_{0}^{T}\big(\sum_{j=1}^{\infty}1_{J}(t,j)\|f_{j}\|_{L_{p}}\big)^{q}dt +2^q\int_{0}^{T}\big(\sum_{j=1}^{\infty}1_{J^{c}}(t,j)  \phi(2^{2j})^{-\frac{1}{\alpha}}  t^{-1}\|f_{j}\|_{L_{p}}\big)^{q}dt,
\end{aligned}
\end{equation*}
where $J=\{(t,j)|  \phi(2^{2j})^{-\frac{1}{\alpha}}  t^{-1}\geq 1 \}$. By H\"older's inequality, 
\begin{equation*}
\begin{aligned}
\int_{0}^{T}\big(\sum_{j=1}^{\infty}1_{J}\|f_{j}\|_{L_{p}}\big)^{q}dt
&=\int_{0}^{T}\big(\sum_{j\in J(t)}  \phi(2^{2j})^{\frac{a}{\alpha}}    \phi(2^{2j})^{-\frac{a}{\alpha}}  \|f_{j}\|_{L_{p}}\big)^{q}dt
\\
&\leq \int_{0}^{T}\big(\sum_{j\in J(t)}  \phi(2^{2j})^{-\frac{aq'}{\alpha}}  \big)^{\frac{q}{q'}}\big(\sum_{j\in J(t)}  \phi(2^{2j})^{\frac{aq}{\alpha}}  \|f_{j}\|^{q}_{L_{p}}\big)dt,
\end{aligned}
\end{equation*}
where $a\in(-\frac{1}{q},0)$, $q'=\frac{q}{q-1}$, and $J(t)=\{j=1,2,\dots|(t,j)\in J\}$.

Fix $t>0$ and let $j_{0}(t)$ be the largest integer such that $\phi(2^{2j_{0}})^{-\frac{1}{\alpha}}t^{-1}\geq 1$. Then the above summation on $J(t)$ is the summation over $\{j\leq j_{0}\}$.  Moreover, by \eqref{phiratio} we have
\begin{equation}\label{eqn 06.15.14:27}
2^{-2}\leq \frac{\phi(2^{2j-2})}{\phi(2^{2j})}  \leq   c^{-1} 2^{-2\delta_{0}}.
\end{equation}
This yields that
\begin{equation*}\label{eqn 06.15.15:13}
\sum_{j\in J(t)} \phi(2^{2j})^{-\frac{aq'}{\alpha}}   \leq N(q,c,\delta_{0}) t^{aq'}.
\end{equation*}
Hence, we have
\begin{equation}\label{middle_of_Besov_bound_2_2}
\begin{aligned}
\int_{0}^{T}\big(\sum_{j=1}^{\infty}1_{J}\|f_{j}\|_{L_{p}}\big)^{q}dt
&\leq N\sum_{j=1}^{\infty} \int_{0}^{  \phi(2^{2j})^{-\frac{1}{\alpha}}  }t^{aq}  \phi(2^{2j})^{\frac{aq}{\alpha}}  \|f_{j}\|^{q}_{L_{p}}dt
\\
&\leq N \sum_{j=1}^{\infty}  \phi(2^{2j})^{-\frac{1}{\alpha}}  \|f_{j}\|^{q}_{L_{p}}.
\end{aligned}
\end{equation}
By H\"older's inequality again, for $b\in(-1,-\frac{1}{q})$ and $q'=\frac{q}{q-1}$,
\begin{equation*}
\begin{aligned}
\int_{0}^{T}&\big(\sum_{j=1}^{\infty}1_{J^{c}}  \phi(2^{2j})^{-\frac{1}{\alpha}} t^{-1}\|f_{j}\|_{L_{p}}\big)^{q}dt
\\
&=\int_{0}^{T}\big(\sum_{j\notin J(t)} \phi(2^{2j})^{-\frac{b}{\alpha}}  \phi(2^{2j})^{\frac{b}{\alpha}}  \phi(2^{2j})^{-\frac{1}{\alpha}}  t^{-1}\|f_{j}\|_{L_{p}}\big)^{q}dt
\\
&\leq \int_{0}^{T}t^{-q}\big(\sum_{j\notin J(t)} \phi(2^{2j})^{-\frac{(b+1)q'}{\alpha}} \big)^{\frac{q}{q'}}\big(\sum_{j\notin J(t)}{ \phi(2^{2j})^{\frac{bq}{\alpha}}  }\|f_{j}\|^{q}_{L_{p}}\big)dt.
\end{aligned}
\end{equation*}
The summation over $J(t)^{c}$ is the summation over $\{j>j_{0}\}$. Hence by \eqref{eqn 06.15.14:27} we have
\begin{equation*}
\sum_{j\notin J(t)} \phi(2^{2j})^{-\frac{(b+1)q'}{\alpha}}   \leq N(q) t^{(b+1)q'}.
\end{equation*}
Therefore, we have
\begin{equation}\label{middle_of_Besov_bound_2_3}
\begin{aligned}
\int_{0}^{T}\big(\sum_{j=1}^{\infty}1_{J^{c}}  \phi(2^{2j})^{-\frac{1}{\alpha}}  t^{-1}\|f_{j}\|_{L_{p}}\big)^{q}dt
&\leq N \sum_{j=1}^{\infty}\int_{ \phi(2^{2j})^{-\frac{1}{\alpha}} }^{\infty}t^{-q}t^{(b+1)q}{ \phi(2^{2j})^{\frac{bq}{\alpha}} }\|f_{j}\|^{q}_{L_{p}}dt
\\
&=N \sum_{j=1}^{\infty}\phi(2^{2j})^{-\frac{1}{\alpha}}  \|f_{j}\|^{q}_{L_{p}}.
\end{aligned}
\end{equation}
Combining \eqref{eqn 07.09.1}, \eqref{middle_of_Besov_bound_2_2} and \eqref{middle_of_Besov_bound_2_3}, we have \eqref{Besov_bound_of_q}. The theorem is proved.
\end{proof}

\begin{lemma}\label{thm 06.23.14:27}
Let $0<\alpha<1$, $1<q,p<\infty$, $\gamma\in\bR$ and $T<\infty$. Then, for any  $u_{0}\in B^{\phi,\gamma+2-2/\alpha q}_{p,q}$ equation  \eqref{eqn 06.15.16:05} has a solution $u\in\mathbb{H}^{\alpha,\phi,\gamma+2}_{q,p}(T)$ satisfying
\begin{equation}\label{eqn 06.23.11:02}
\|u\|_{\bH_{q,p}^{\alpha,\phi,\gamma+2}(T)}  \leq N \|u_{0}\|_{B^{\phi,\gamma+2-2/\alpha q}_{p,q}},
\end{equation}
where the constant $N$ depends only on $\alpha,d,p,q,\phi,\gamma$, and $T$.
\end{lemma}

\begin{proof}
By Remark \ref{rmk 06.23.11:13}, and Lemma \ref{basicproperty} (iii),  it is enough to prove  the case $\gamma=-2$. 

If $u_{0}\in C^{\infty}_{p}(\R^{d})$, then we define
\begin{equation}\label{eqn 06.23.11:25}
u(t,x)=\int_{\R^{d}}q(t,x-y)u_{0}(y)dy.
\end{equation}
Then $u\in C^{\alpha,\infty}_{p}([0,T]\times\bR^{d})$ and all the claims of the lemma hold with $u$ due to 
 Lemma \ref{zero converge}(ii) and Lemma \ref{Besov_bound_2}.
 
  In general, for $u_{0}\in B^{\phi,-2/\alpha q}_{p,q}$ we take a sequence $u^n_0\in \cS$ such that $u^n_0\to u_0$ in $B^{\phi,-2/\alpha q}_{p,q}$, and  we define $u_{n}\in C^{\alpha,\infty}_{p}([0,T]\times\bR^{d})\cap \mathbb{H}^{\alpha,\phi,0}_{q,p}(T)$ corresponding to $u^n_0$ using  \eqref{eqn 06.23.11:25}.  Then inequality \eqref{eqn 06.23.11:02} applied to $u_n-u_m$ shows that $u_n$ is a Cauchy sequence in $\mathbb{H}^{\alpha,\phi,0}_{q,p}(T)$. Finally one gets the claims of the lemma by considering the limit. The lemma is proved.
\end{proof}

\begin{corollary}
   \label{cor 8.3.1}
   $u\in{\bH_{q,p}^{\alpha,\phi,\gamma+2}(T)}$ and $u(0,\cdot)=0$ if and only if  $u\in{\bH_{q,p,0}^{\alpha,\phi,\gamma+2}(T)}$. 
    \end{corollary}
    \begin{proof}
    We only prove ``only if" part. The ``if"  part is obvious by definition.  Suppose  $u\in{\bH_{q,p}^{\alpha,\phi,\gamma+2}(T)}$ and $u(0,\cdot)=0$. Then there exists a defining sequence $u_n\in C^{\alpha,\infty}_{p}([0,T]\times \bR^d)$ of $u$ such that $u_n(0,\cdot)\in C^{\infty}_p$ and 
\begin{equation*}
u_n(0,\cdot) \to 0 \text{ in } B^{\phi,\gamma+2-2/\alpha q}_{p,q}.
\end{equation*}
By Lemmas \ref{zero converge} and  \ref{thm 06.23.14:27}, we can choose $v_n\in C^{\alpha,\infty}_{p}([0,T]\times \bR^d)$ such that $v_n(0,\cdot)=u_n(0,\cdot)$ and
$$
\|v_n\|_{{\bH_{q,p}^{\alpha,\phi,\gamma+2}}(T)} \leq N \|u_{n}(0,\cdot)\|_{B^{\phi,2-2/\alpha q}_{p,q}} \to 0
$$
as $n\to \infty$. This implies that $u_n-v_n$ is also a defining sequence of $u$, and therefore we  have  $u\in \bH_{q,p,0}^{\alpha,\phi,2}(T)$ because $(u_n-v_n)(0,\cdot)=0$.  
         \end{proof}

\mysection{Proof of Theorem \ref{main theorem}}

Due to Remark \ref{rmk 06.23.11:13} and Lemma \ref{basicproperty} (iii), we only need to prove  case $\gamma=0$.

\vspace{1mm}

\textbf{Step 1} (Uniqueness).  Let  $u\in \bH_{q,p}^{\alpha,\phi,2}(T)$ be a solution to equation \eqref{mainequation} with $f=0$ and $u_{0}=0$.  Then by Corollary \ref{cor 8.3.1},   $u\in \bH_{q,p,0}^{\alpha,\phi,2}(T)$.  Hence, by Lemma \ref{basicproperty} (ii), there exists ${u}_n\in C_c^\infty (\bR^{d+1}_+)$ such that ${u}_n \to u$ in $\bH_{q,p}^{\alpha,\phi,2}(T)$.  Due to Lemma \ref{u=qfsolution}, it also holds that 
\begin{equation}\label{eqn 03.25.14:49}
u_n(t,x)=\int_0^t \int_{\bR^d} q_{\alpha,1}(t-s,x-y)f_n(s,y) dyds,
\end{equation}
where $f_n:=\partial_t^\alpha u_n - \phi(\Delta) u_n$.  Note
\begin{eqnarray*}
\|f_n\|_{\bL_{q,p}(T)}&=&\|\partial_t^\alpha (u_n-u) - \phi(\Delta) (u_n-u)\|_{\bL_{q,p}(T)}\\
&\leq& \|\partial_t^\alpha u_n-\partial_t^\alpha u \|_{\bL_{q,p}(T)}+ \|\phi(\Delta) u_n- \phi(\Delta) u\|_{\bL_{q,p}(T)} \to 0
\end{eqnarray*}
as $n\to \infty$.  Thus  by Minkowski's inequality  we have
$$
\|u_{n}\|_{\bL_{q,p}(T)} \leq N(T) \|f_{n}\|_{\bL_{q,p}(T)}.
$$
Letting $n\to\infty$, we get $u=0$  since $\bH_{q,p}^{\alpha,\phi,2}(T) \subset \bL_{q,p}(T)$.

{\textbf{Step 2}} (Existence and estimate with zero initial condition). 
First assume $f\in C_c^\infty(\bR^{d+1}_+)$, and let 
\begin{equation}\label{eqn 07.06.16:36}
u(t,x):=\int_0^t \int_{\bR^d} q_{\alpha,1}(t-s,x-y)f(s,y) dyds.
\end{equation}

Using Remark \ref {Hvaluedconti} and the integrability of $q_{\alpha,1}$ one can easily check 
 $D^m_x u$, $\phi(\Delta)D^m_x u$ $\in C([0,T];L_p)$, and therefore $u\in C^{\alpha,\infty}_p([0,T]\times\bR^d)$. Also, by Lemma \ref{u=qfsolution},  $u$ satisfies equation \eqref{mainequation} with $u(0,\cdot)=0$. 
 
 Now we show \eqref{mainestimate}. By Minkowski's inequality and Corollary \ref{qintegrable} (i), one has
\begin{equation*}\label{eqn 03.25.15:06}
\|u\|_{\bL_{q,p}(T)} \leq N(T) \|f\|_{\bL_{q,p}(T)}.
\end{equation*}
Also the estimate
$$
\|\phi(\Delta)u\|_{\mathbb{L}_{q,p}(T)}\leq N \|f\|_{\mathbb{L}_{q,p}(T)}
$$
follows from Theorem \ref{qpestimate}. These two inequalities with  \eqref{eqn 03.25.15:03} lead to \eqref{mainestimate} and \eqref{mainestimate-11}.

For general $f$,  we take a sequence of functions   $f_{n}\in \Ccinf(\R^{d+1}_{+})$  such that $f_n \to f$ in $\bL_{q,p}(T)$. Let $u_{n}$  denote the solution to equation  \eqref{eqn 07.06.16:36} with $f_{n}$ in place of $f$. Then \eqref{mainestimate} applied to $u_m-u_n$ shows that $u_{n}$ is Cauchy in 
$\bH^{\alpha,\phi,2}_{q,p,0}(T)$. By taking $u$ as the limit of $u_{n}$ in $\bH^{\alpha,\phi,2}_{q,p,0}(T)$, we find that $u$ satisfies the equation $\partial_t^\alpha u = \phi(\Delta)u +f$, and  \eqref{mainestimate} and  \eqref{mainestimate-11} also hold for $u$.

\textbf{Step 3} (Existence and estimate with nonzero initial condition).
Let $v\in\mathbb{H}^{\alpha,\phi,\gamma+2}_{q,p}(T)$ denote the solution to the homogeneous equation taken from Lemma \ref{thm 06.23.14:27}, and let $u\in\mathbb{H}^{\alpha,\phi,\gamma+2}_{q,p,0}(T)$  be taken from Step 2. Then $\bar{u}:=v+u\in\mathbb{H}^{\alpha,\phi,\gamma+2}_{q,p}(T)$ satisfies \eqref{mainequation}, and \eqref{mainestimate}  also holds. The theorem is proved.

\end{document}